\documentclass[10pt]{article}

\usepackage{geometry}
\usepackage{amsmath}
\usepackage{amssymb}
\usepackage{amsthm}
\usepackage[latin1]{inputenc}
\usepackage{eurosym}
\usepackage[dvips]{graphics}
\usepackage{graphicx}
\usepackage{epsfig}
\usepackage{hyperref}
\usepackage{ifthen}
\makeindex

\newcommand{\supp}{\operatorname{supp}}

\newcommand{\cl}{\operatorname{cl}}

\newcommand{\Rr}{{\mathbb{R}}}

\newcommand{\Zz}{{\mathbb{Z}}}

\newcommand{\Tt}{{\mathbb{T}}}

\newcommand{\Hh}{{\overline{H}}}

\newcommand{\Mm}{{\mathcal{M}}}

\theoremstyle{definition}

\newtheorem{theorem}{Theorem} \newtheorem{corollary}[theorem]{Corollary}
\newtheorem{remark}{Remark}
\newtheorem{lemma}[theorem]{Lemma}
 \newtheorem{proposition}[theorem]{Proposition}

\theoremstyle{definition}

\begin{document}

\title{Mather problem and viscosity solutions in the stationary setting}

\author{Diogo A.  Gomes and Elismar R. Oliveira}
\date{\today}
\maketitle

\thanks{D. Gomes was
partially supported by CAMGSD/IST through FCT Program POCTI - FEDER
and by grants POCI/FEDER/MAT/55745/2004, DENO/FCT-PT (PTDC/EEA-ACR/67020/2006). Elismar R. Oliveira was supported by POCI - FEDER/MAT/55745/2004 program.}

\begin{abstract}
In this paper we discuss the  Mather problem for stationary Lagrangians, that is 
Lagrangians $L:\Rr^n\times \Rr^n\times \Omega\to \Rr$, 
where $\Omega$ is a compact metric space on which $\Rr^n$ acts through an action which leaves $L$ invariant. 
This setting allow us to generalize the standard Mather problem for quasi-periodic and almost-periodic Lagrangians.
Our main result is 
the existence of stationary Mather measures invariant under the Euler-Lagrange flow which are supported in a graph. 
We also obtain several estimates for viscosity solutions of Hamilton-Jacobi equations for the discounted cost 
infinite horizon problem. 
\end{abstract}
 
\section{Introduction}
Let $M$ be  a complete compact manifold, and $L:TM \to \Rr$ a $C^{3}$ Lagrangian, fiberwise strictly convex and  coercive.
A probability measure on $TM$ is called holonomic if 
\[
\int_{TM} v\cdot D\varphi d\mu=0,
\]
for all $\varphi \in C^1(M)$.
A central
 result in Aubry-Mather theory
\cite{Mane5} (see also \cite{F6}), 
is the fact that any holonomic probability measure $\mu$ on $TM$ 
which minimizes the action $\int_{TM} L d\mu$ is supported on a Lipschitz graph and is
invariant under the Euler-Lagrange flow.
Certain results in Aubry-Mather theory have been extended for non-compact manifolds, see for instance \cite{FathiMad} or \cite{Maderna},
but as far as the authors know, there is in the literature no satisfactory construction of Mather measures for general non-compact manifolds.

In this paper, rather than considering Lagrangians on the tangent bundle of compact manifolds, such as in 
the original paper of Mather \cite{Mather}, we consider
Lagrangians defined on  $\Rr^n\times \mathbb{R}^{n}\times\Omega$, where $\Omega$ is a suitable compact metric space
on which $\Rr^n$ acts trough an action $\tau_x$. 
The main result of this paper is Theorem \ref{InvarianceFlowMeasures}, in which
we establish the existence of stationary Mather measures invariant under the Euler-Lagrange flow. 

%The motivation for the study of the stationary Mather problem is to understand the global behavior of orbits of Lagrangians that depend (in a suitable way) on a parameter %$\omega \in \Omega$, where $\Omega$ is a compact metric space.
Stationary ergodic problems were considered in \cite{LS} in the context of homogenization of random stationary ergodic Hamilton-Jacobi equations. The authors (in particular DG) are thankful to 
several enlightening discussions with P. Souganidis on this issue.  Generalized Mather measures for stationary ergodic problems were
also considered in the homogenization setting in \cite{GV}. The stationary ergodic setting was consider in \cite{DS} where the 
construction of critical (or critical approximate) viscosity solutions of Hamilton-Jacobi equations is carried out in 
detail for the one-dimensional case.  

A simple example (taken from \cite{LS}) which illustrates the main difficulties in the stationary setting is the Lagrangian
\[
 L=\frac{|v|^2}{2}-\cos(x+\omega_1)-\cos(\sqrt{2} x+\omega_2). 
\]
Consider $\omega\in \Rr^2/\Zz^2\equiv \Tt^2$ as a fixed parameter.
It would be natural, 
as in Mather's problem, to look for probability measures $\mu$ on $\Rr^n\times \Rr^n$ which minimize the action 
\begin{equation}
\int_{\Rr^n\times \Rr^n} L d\mu
\label{choose} 
\end{equation}
under the holonomy constraint
\[
\int_{\Rr^n\times \Rr^n} v \cdot D_x\varphi d\mu=0, 
\]
for all $\varphi$ of class $C^1$, bounded with bounded derivatives. This problem can be solved explicitly, and in fact we
have the following two cases: if there exists a solution $\bar x$ to the overdetermined system
\[
\bar x+\omega_1=2\pi n, \qquad \sqrt{2} \bar x+\omega_2=2\pi n, 
\]
for some $n\in \Zz$,  
the Mather measure on $\Rr \times \Rr $ is simply $\mu_0=\delta_{\bar x}(x) \delta_0(v)$; otherwise there does not exist a Mather measure since $L > -1$ for all $(x,v)$, and the infimum in \eqref{choose} is easily shown to be -1.

To overcome these issues, which are due to the lack of compactness of $\Rr^n$, we will instead define
stationary Mather measures as measures on $(v, \omega)\in \Rr^n\times \Omega$, which minimize the action and satisfy a suitable holonomy
condition. It turns out that if $\Omega$
is compact and the Lagrangian satisfies certain stationarity hypothesis this is the natural way to generalize Mather measures. Before
proceeding, we must make precise our framework.

Let $\Omega$ be a compact metric space, and let
$L=L(x, v, \omega): \mathbb{R}^{n}  \times \mathbb{R}^{n} \times \Omega \to \mathbb{R}$  be a continuous Lagrangian, $C^3$ in the first two coordinates.
The Lagrangian $L$ is also required to be strictly convex and superlinear on the velocity $v$, and nonnegative.  In our setting,
this last condition can be achieved without changing of the nature the problem by  adding a constant to $L$. We assume further that 
\begin{equation}
\label{lipep}
L(x+y, v, \omega)-L(x, v, \omega)\leq |y|\left(C+C L(x, v, \omega)\right). 
\end{equation}

We suppose that there exists an action $\tau: \Omega \times \mathbb{R}^{n} \to  \Omega$ which is continuous, satisfies the semigroup property
$$\tau_{x+y}\omega=\tau_{x } \tau_{y}\omega \text{ and }\tau_{0}(\cdot)=Id.$$
Since $\Omega$ is compact and the action is continuous, the action is
uniformly transitive\footnote{The authors are grateful to Albert Fathi that pointed out to us that uniform transitivity holds
under the compactness assumption.} in the following sense:
  $$\forall \varepsilon >0, \exists M>0, \forall  \omega_{1},  \omega_{2} \in \Omega, \exists z \in \mathbb{R}^{n}, \text{ such that } |z| <M \text{, and } d(\tau_{z} \omega_{1},  \omega_{2}) < \varepsilon.$$

A first example of such an action is the following: we take $\Omega=\Tt^d$, the $d$-dimensional torus, let $n<d$ and we will
construct an action
 $\tau : \mathbb{R}^n  \times \mathbb{T}^{d}  \to \mathbb{T}^{d}$. To start with, we identify the torus $\Tt^d$ with its universal covering
$\Rr^d$, and consider a constant coefficient $d\times n$ matrix $A$. Assume that $\{Ax: x\in \Rr^n\}$ is dense in $\Tt^n$.
Then we define
\[
\tau_{x}\omega=\omega+A x.
\]

A second example is the following. We take $\Omega$ to be the space of all sequences $\omega=(\omega_k)$ on $\Tt^1$, endowed with the following metric:
\[
d(\omega, \tilde \omega)=\sum_{k=1}^\infty 2^{-k} |\omega_k-\tilde \omega_k|. 
\]
It is simple to verify that with this distance the space $\Omega$ is compact. A sequence $\lambda$
of real numbers
 is called irrational if for any $N$ the vector $(\lambda_1, \hdots, \lambda_N)$
is is linearly independent over the integers. Let $\bar \lambda$ be an irrational sequence. Define the following action from $\Rr$ into $\Omega$ by
\[
\tau_x \omega=\omega + x \bar \lambda. 
\]
This action is also uniformly transitive. 

%{\bf perhaps add the proof here}

A  function $\varphi :\mathbb{R}^{n}\times\Rr^n \times \Omega \to \mathbb{R}$,  is stationary if $$\varphi(x+y, v, \omega) = \varphi(x,v, \tau_{y} (\omega)), \; \forall x, y \in \mathbb{R}^{n}, \; \omega \in \Omega.$$
We assume that the Lagrangian $L$ is stationary. 

Denote 
\begin{align*}
C^{1}_{s}(\mathbb{R}^{n} \times \Omega)= \{ &\varphi :\mathbb{R}^{n} \times \Omega \to \mathbb{R}, \; \text{stationary, }  \; C^{1} \;  \text{in the first variable,}  \text{ continuous in} \; \\& \omega \text{, and such that } D_{x}\varphi(0, \omega)  \; \text{is continuous in} \; \omega\}, 
\end{align*}
with an analogous definition for $C^1_s(\Rr^n\times \Rr^n\times \Omega)$. 

If the action is given as in the first example by
$
\tau_{x}\omega=\omega+A x,
$ 
given $\psi: \Tt^d\to \Rr$, the function $\varphi(x, \omega)=\psi(\omega+Ax)$ is stationary, and, furthermore, $\varphi \in C^1_s$ if
$\psi$ is $C^1$. In the second example we can construct an example of a stationary function in the following way: let $\psi_k: \Tt\to \Rr$ be a sequence of periodic functions
uniformly bounded in $k$. 
Let
\[
\varphi(x, \omega)=\sum_k \psi_k(\omega_k+\bar\lambda_k x) 2^{-k} \frac 1 {1+|\bar \lambda_k|}
\]
Furthermore, if $\psi_k$ is $C^1$ and its derivatives are uniformly bounded in $k$, $\varphi\in C^1_s$. 

To motivate the stationary Mather problem,
let $x(t)$ be a globally Lipschitz trajectory on $\Rr^n$. Let $\omega_0\in \Omega$ is an arbitrary point.
Consider 
ergodic averages to define an occupation measure $\mu$ on $\Rr^n\times \Omega$ corresponding to $x(\cdot)$ in the following way
\begin{align*}
\lim_{T \to \infty} \frac{1}{T} \int_{0}^{T} \phi(x,\dot{x}, \omega_0)  dt &= \lim_{T \to \infty} \frac{1}{T} \int_{0}^{T} \phi(0,\dot{x},\tau_{x}\omega_0)  dt  \equiv \int_{\mathbb{R}^{n} \times \Omega}  \phi (0, v, \omega) d\mu,
\end{align*}
where the limit is taken trough an appropriate sequence.  Of course, the measure $\mu$ could depend on the point $\omega_0$ 
or the sequence through which the limit is taken. Nevertheless, 
such probabilities $\mu$, satisfy an integral constraint, the holonomy condition:
\begin{equation}
\label{hnom}
\int_{\mathbb{R}^{n} \times \Omega}  v \cdot D_{x}\varphi(0, \omega)  d\mu =0,
\end{equation}
for any stationary function $\varphi \in C^{1}_{s}(\mathbb{R}^{n}\times\Omega)$. 

The stationary Mather problem can be formulated as follows: minimize 
$$\int_{\mathbb{R}^{n} \times \Omega} L(0, v, \omega) d\mu (v, \omega),$$
over all probability measures that satisfy the holonomy constraint \eqref{hnom}.
A minimizing measure for this problem is called a stationary Mather measure. 
A similar problem arises also in \cite{GV} for the homogenization of Hamilton-Jacobi equations. 

Let $\gamma:\mathbb{R}^{n}  \to \mathbb{R}$ be a positive function 
such that \begin{equation}
\label{gammadef}
\lim_{\vert v \vert \to \infty} \frac{\vert v \vert}{\gamma(v)} =0, \; \text{and} \lim_{\vert v \vert \to \infty} \frac{L(0, v, \omega)}{\gamma(v)} =+\infty, 
\end{equation}
where the last limit is uniform in $\omega \in \Omega$ by compactness.
We denote by $C_\gamma^0(\mathbb{R}^{n} \times \Omega)$ the set 
 of the continuous
functions $\phi$ with
$$\|\phi\|_\gamma=\sup_{\mathbb{R}^{n} \times \Omega } \frac{|\phi(v, \omega)|}{\gamma(v)}<\infty, \qquad
\lim_{|v|\rightarrow \infty}\frac{ \vert \phi(v, \omega)  \vert}{\gamma(v)}\rightarrow 0.$$

We will need also to consider the discounted Mather problem, see \cite{GomesSelection} for a discussion of related generalizations of Mather's problem. For that, let $\alpha$ be a positive number. 
Consider the operator $A: C^{1}_{s}(\mathbb{R}^{n} \times \Omega)  \to C_\gamma^0(\mathbb{R}^{n} \times \Omega)$ given by
$$ \varphi \to A^{v} \varphi (\omega) =v \cdot D_{x}\varphi(0, \omega) - \alpha \varphi(0, \omega).$$

The discounted stationary Mather problem consists in minimizing
$$\int_{\mathbb{R}^{n} \times \Omega} L(0, v, \omega) d\mu (v, \omega)$$ 
over all probability measures that satisfy
the discounted holonomy constraint 
\begin{equation}
\label{dhnom}
\int_{\mathbb{R}^{n} \times \Omega} A^{v} \varphi (\omega)  d\mu (v, \omega)=- \alpha \int_{ \Omega} \varphi(0, \omega) d\nu (\omega),
\end{equation}
 for all $\varphi \in C^{1}_{s}(\mathbb{R}^{n} \times \Omega)$.
A minimizing probability measure for this problem is called a discounted stationary Mather measure. The measure $\nu$ is called the trace of $\mu$.
If $\alpha=0$ we call these measures stationary Mather measures. 

The main result of this paper 
is the construction of
 stationary Mather measures invariant under the Euler-Lagrange flow. Usually, this flow is defined in $\Rr^n\times \Rr^n$. However, since
the stationary Mather measures are measures on $\Rr^n\times \Omega$ we must now discuss the natural extension of the Euler-Lagrange flow to this space. 

Given a stationary vector field $W: \mathbb{R}^{n} \times  \mathbb{R}^{n} \times \Omega \to \mathbb{R}^{n} \times \mathbb{R}^{n}$, let $\Phi=(\Phi_{1}, \Phi_{2}):  \mathbb{R}\times \mathbb{R}^{n}  \times \mathbb{R}^{n} \times \Omega \to  \mathbb{R}^{n} \times  \mathbb{R}^{n}$ be its flow. 
We define the flow $\Psi:\Rr\times \Rr^n\times \Omega\to \Rr^n\times \Omega$ 
induced by $W$ in $\mathbb{R}^{n} \times \Omega$ as 
$$\Psi(t, v, \omega)=(\Phi_{2}(t, 0, v, \omega), \tau_{\Phi_{1}(t, 0, v, \omega)} \omega).$$
We denote by $C^1_b(\Rr^n\times \Omega)$ the set of bounded continuous functions $\phi(v, \omega)$ in $\Rr^n\times \Omega$ such that $D_v\phi(v, \omega)$ is also continuous 
and bounded. 
A measure $\mu$ is invariant under the flow $\Psi$ if, 
$$\int_{\mathbb{R}^{n} \times \Omega} \phi ( \Psi(t, v, \omega) ) d\mu(v, \omega) = \int_{\mathbb{R}^{n} \times \Omega} \phi(v, \omega) d\mu(v, \omega),$$
for all $\phi \in C^1_b(\mathbb{R}^{n} \times \Omega)$ and for all $t \in \mathbb{R}$. 

Let $\mu$ be a  measure in $\mathbb{R}^{n} \times \Omega$ and  $W: \mathbb{R}^{n} \times  \mathbb{R}^{n} \times \Omega \to \mathbb{R}^{n} \times \mathbb{R}^{n} $ be a stationary vector field in $\mathbb{R}^{n} \times \mathbb{R}^{n}$.  Then $\mu$ is invariant under the flow induced by $W$ in $\mathbb{R}^{n} \times \Omega$, if and only if, 
\begin{equation} \int_{\mathbb{R}^{n} \times \Omega} \nabla \hat{\phi}(0, v, \omega) \cdot W(0, v, \omega) d\mu(v, \omega) = 0,\label{InvCond} \end{equation}
(where the gradient in the previous formula is taken both in $x$ and $v$) for all  $\hat{\phi} \in \hat{C}_\gamma^0(\mathbb{R}^{n} \times \mathbb{R}^{n} \times \Omega)$.
A proof for this classical fact for the case of vector fields on a manifold $M$ can be found, for instance,
in \cite{BG}. The proof in our setting follows exactly along
the same lines and we will omit it. 

%We consider the stationary Euler-Lagrange equation
%$$  \frac{d  }{d t} D_vL(x, v, \omega) = D_xL(x, v, \omega),$$
%for each $\omega \in \Omega$. Then, we have a $\omega$-parametric Lagrangian vector field $W^{L}$, that is given by:
%$$ W^{L}=
%\begin{cases}
%X^{L}(x, v, \omega)= v      \\
%Y^{L}(x, v, \omega)= (D^2_{vv}L)^{-1}(D_xL - D_{xv}L v).
%\end{cases} $$

In this paper we will need to consider the discounted Lagrangian $L_\alpha\equiv e^{-\alpha t} L(x, v)$. The corresponding Euler-Lagrange equation is
\begin{equation}
\label{discountel}
  \frac{d  }{d t} D_vL(x, v, \omega) = D_xL(x, v, \omega)+\alpha D_vL,
\end{equation}
for each $\omega \in \Omega$.
For $\alpha=0$ we obtain the usual Euler-Lagrange equations. 
We have a $\omega$-parametric Lagrangian vector field $W^{L_\alpha}$, that is given by:
$$ W^{L_\alpha}=
\begin{cases}
X^{L_\alpha}(x, v, \omega)= v      \\
Y^{L_\alpha}(x, v, \omega)= (D^2_{vv}L)^{-1}(D_xL +\alpha D_vL- D_{xv}L v).
\end{cases} $$
We say that a  measure $\mu$ in $\mathbb{R}^{n} \times \Omega$ is invariant under the Euler-Lagrange flow if it is invariant under the flow $\Psi^\alpha$ induced by $W^{L_\alpha}$ in $\mathbb{R}^{n} \times \Omega$.

The outline of this paper is as follows: in section \ref{dualsec} we describe briefly the duality theory for the stationary Mather problem and its connections with viscosity solutions
of Hamilton-Jacobi equations. The proofs of some the results, since they are standard, are outlined for completeness in appendix \ref{apA}.
In section \ref{sfc} we make some formal computations in the spirit of \cite{EGom1}. These computations suggest that for certain discounted stationary Mather
measures  one may be able to extend the regularity results in \cite{EGom1}. Holonomic discounted stationary Mather measures are constructed in section \ref{hdmm}.  
Using these measures we obtain regularity results for viscosity solutions
in section \ref{graph}.  These imply that the discounted stationary Mather measures are supported in a (partially) Lipschitz graph
whose Lipschitz constant is independent of the discount factor $\alpha$. Finally in the last section
we construct stationary Mather measures invariant under the Euler-Lagrange flow. 

%In section \ref{hnmv} we extend ideas from \cite{BG} to construct certain holonomy preserving variations which in turn are used in section \ref{invar} to 
%prove that this stationary Mather measure is invariant under the Euler-Lagrange flow.

\section{Duality and viscosity solutions}
\label{dualsec}

The stationary Mather problem is an infinite dimensional linear programming problem. As usual in these problems (see \cite{GomesSelection}, for instance), the duality theory plays an important role and will be
developed in this section. 
\begin{theorem}\label{Duality}
Let $\nu$ be a probability measure on $\Omega$ and $\alpha \geq 0$. Define
\begin{equation}
\label{hbaralf}
\Hh_\alpha=\inf \int_{\mathbb{R}^{n} \times \Omega} L(0, v, \omega) d\mu (v, \omega),
\end{equation}
where the infimum is taken over all probability measures on $\mathbb{R}^{n} \times \Omega$ which satisfy the discounted holonomy condition 
\eqref{dhnom}.
Let
\[
\mathcal{H}(\varphi, x, \omega)= \sup_{ v \in\mathbb{R}^{n}}  ( - A^{v}(\varphi)(x, \omega) - L(x, v, \omega))= H(x, D_{x}\varphi(x, \omega), \omega) +  \\ \alpha \varphi(x, \omega),
\]
with
$H(x, p, \omega) = \sup_{ v \in\mathbb{R}^{n} } ( - p\cdot v  - L(x, v, \omega))$.

Then, the infimum in \eqref{hbaralf} is achieved at some probability measure $\mu$ satisfying \eqref{dhnom} and furthermore
\begin{equation}
\label{dualf}
\Hh_\alpha = - \inf_{ \varphi \in C^{1}_{s}  } \sup_{ \Omega}  \left\lbrace -\alpha \int_{\Omega} \varphi  d\nu  +  \mathcal{H}(\varphi, 0, \omega) \right\rbrace. 
\end{equation}
%thus we get
%$$\bar{H}_{\alpha} = -  \inf_{ \varphi \in C^{1}_{s}}  \sup_{ \omega \in \Omega} 
%- \alpha \int_{ \Omega} \varphi(0, \omega) d\nu (\omega) + \mathcal{H}(\varphi, 0,\omega),$$
%where $\displaystyle \bar{H}_{\alpha}=\min \int_{\mathbb{R}^{n} \times \Omega} L(0, v, \omega) d\mu (v, \omega)$.
\end {theorem}

The proof of this Theorem is similar to analogous results in \cite{GomesSelection}, for instance. For completeness, however, we present the proof in the Appendix \ref{apA}.

%\begin{remark}\label{mathermeasures}
%From the minimum in Theorem~\ref{Duality} we get the existence of stationary Mather measures, that is, solutions of the stationary Mather problem, since
%the Legendre-Fenchel-Rockafellar duality theorem, which is used to establish \eqref{dualf} asserts that the left hand side of \eqref{dualf} is
%indeed a minimum. 
%\end{remark}

In this paper we will need to consider viscosity solutions to the equation
\begin{equation}
\mathcal{H}^\alpha(u,0,\omega)\equiv H(0, D_{x}u_{\alpha}(0, \omega), \omega) + \alpha u_{\alpha}(0, \omega)=0.
\label{DiscHJE}
\end{equation}
As in the standard Mather problem, viscosity solutions yield important 
information concerning the value of the variational problem \eqref{hbaralf}, and help characterize the support of the measure.

Before we proceed, we make some remarks concerning the regularization by convolution of stationary functions. 
\begin{remark}\label{viscsolsmooth}
To approximate a stationary function $u : \mathbb{R}^{n} \times \Omega \to \mathbb{R}$ by
smooth stationary functions we are going to use a 
convolution with a standard mollifier $\eta^{\varepsilon} : \mathbb{R}^{n}   \to \mathbb{R}$, that is, $\eta$ compactly supported, 
$\eta^{\varepsilon} (x)= \frac{1}{\varepsilon} \eta(\frac{x}{\varepsilon})$, 
and $\int_{\mathbb{R}^{n}} \eta(x) dx=1$.
We define the convolution between $u$ and $\eta^{\varepsilon}$ by 
$$u^{\varepsilon}(x, \omega)=\int_{\mathbb{R}^{n}} u(x, \tau_{y}\omega) \eta^{\varepsilon}(y) dy.$$
Observe that, $u^{\varepsilon} \in C_s^1$. Moreover, we have
$$\frac{\partial u^{\varepsilon}}{\partial  x}(x, \omega)  \cdot v=- \int_{\mathbb{R}^{n}} u(x, \tau_{y}\omega) D_{y}\eta^{\varepsilon}(y)  \cdot v \; dy.$$
\end{remark}

\vspace{0.6cm}

We consider two different types of viscosity solutions for $\mathcal{H}(u , 0, \omega) =\lambda$.
Firstly recall the usual definition of viscosity solution: a function $u : \mathbb{R}^{n} \times \Omega \to \mathbb{R}$,  continuous in $x$ (not necessarily $C^{1}$) for each $\omega \in \Omega$, is a {\it viscosity solution in $x$}
of $\mathcal{H}(u, x,  \omega)=\lambda$ if for each $\omega_{0} \in \Omega$, any $C^{1}$ function $\psi: \mathbb{R}^{n}  \to \mathbb{R}$ and any
$x_{0} \in  \mathbb{R}^{n}$ such that $u(x, \omega_{0})-\psi(x)$ has a strict  local  minimum (resp. maximum) at $x_0$ with $u(x_{0}, \omega_{0})-\psi(x_{0})=0$  we have
$$\mathcal{H}(\psi,  x_{0}, \omega_{0}) \geq \lambda \; \text{(resp. }\leq \lambda).$$
For our purposes we need a modified version of viscosity solution: a stationary (not necessarily $C^{1}$) function $u : \mathbb{R}^{n} \times \Omega \to \mathbb{R}$,
continuous in  $\Omega$,  is a {\it viscosity solution in $\omega$}
of $\mathcal{H}(u, 0, \omega)=\lambda$ if for any  $\varphi \in  C^{1}_{s}(\mathbb{R}^{n} \times \Omega)$ and any point $\omega_{0} \in \Omega$ such that $u(0, \omega)-\varphi(0, \omega)$ has a local  minimum (resp. maximum) at $\omega_0$ with $u(0, \omega_{0})-\varphi(0, \omega_{0})=0$  we have
$$\mathcal{H}(\varphi, 0 ,  \omega_{0}) \geq \lambda \; \text{(resp. }\leq \lambda).$$

\begin{proposition}\label{ViscSolEquiv}
Suppose that  $u : \mathbb{R}^{n} \times \Omega \to \mathbb{R}$  is a viscosity solution in $x$ of $\mathcal{H}(u , 0, \omega) =\lambda$ and 
assume furthermore that $u$ is stationary and continuous in  $\Omega$. Then $u$ is also a viscosity solution in $\omega$ of $\mathcal{H}(u , 0, \omega) =\lambda$. 
\end{proposition}
\begin{proof}
Let $u : \mathbb{R}^{n} \times \Omega \to \mathbb{R}$ be a  viscosity solution in $x$ of $\mathcal{H}(u , 0, \omega) =\lambda$. Consider an
arbitrary function $\varphi \in C^{1}_{s}(\mathbb{R}^{n} \times \Omega)$  and a point $\omega_{0} \in \Omega$ such that $u(0, \omega)-\varphi(0, \omega)$ has a local minimum (resp. maximum) and $u(0, \omega_{0})-\varphi(0, \omega_{0})=0$. Define $\psi(x)=\varphi(x, \omega_{0})$. We claim that $u(x, \omega_{0})-\psi(x)$ has a local minimum (resp. maximum) in $x_{0}=0  \in  \mathbb{R}^{n}$. In fact,  
\begin{align*}
u(x, \omega_{0})-\psi(x)&= u(x, \omega_{0})-\varphi(x, \omega_{0})=  u(0, \tau_{x}\omega_{0})-\varphi(0,\tau_{x}\omega_{0}) \\&\geq  u(0, \omega_{0})-\varphi(0, \omega_{0}) = u(0, \omega_{0})-\psi(0), \quad (\text{resp. } \leq.)
\end{align*}
Then, because $u$ is a viscosity solution in $x$ we have
\[
\mathcal{H}(\psi,0, \omega_0)=\mathcal{H}(\varphi,0, \omega_0)\geq \lambda \quad (\text{resp. }\, \leq \lambda).
\]
\end{proof}
Consider the infinite horizon optimal control problem
\begin{equation}
%\label{DynProPrinc}
u_{\alpha}(x, \omega)=\inf_{x(0)=x} \int_{0}^{+\infty} e^{- \alpha t} L(x(t), \dot{x}(t), \omega) dt,  \end{equation}
where the infimum is taken over all
globally Lipschitz trajectories with initial condition $x(0)=x$.
Then $u_{\alpha} : \mathbb{R}^{n} \times \Omega \to \mathbb{R}$ satisfies the dynamic programing principle
\begin{equation}
u_{\alpha}(x, \omega)=\inf_{x(0)=x} \left( \int_{0}^{T} e^{- \alpha t} L(x(t), \dot{x}(t), \omega) dt + 
e^{- \alpha T} u_{\alpha}(x(T), \omega) \right), \label{DynProgPrinciple}
\end{equation}
among all globally Lipschitz trajectories with initial condition $x(0)=x$.
It is standard, see \cite{Bardi}, that the function $u_\alpha$ is a viscosity solution of $\mathcal{H}(\varphi,  0, \omega)=0$ in $x$.
Furthermore, the optimal trajectories are solutions to the discounted Euler-Lagrange equations \eqref{discountel}. Finally, for $0<t<T$ we have additionally
that $D_xu_\alpha(x(t))$ exists and
\[
\dot x(t)=-D_pH(D_xu_\alpha(x(t)), x(t)). 
\]

The next proposition is also a well known result, see, for instance, \cite{Bardi} for similar results:
%as well as 
%\cite{GomesSelection}:
\begin{proposition}\label{EstabilityViscositySol} For each $\omega$ fixed, let $u_{\alpha}(x, \omega)$ be a viscosity solution (in $x$) of 
\begin{equation}
\label{halfie}
\mathcal{H}^\alpha(u,x,\omega)=
 H(x, D_{x}u_{\alpha}(x, \omega), \omega) + \alpha u_{\alpha}(x, \omega)=0.
 \end{equation}
Then $\alpha u_{\alpha}$ is uniformly bounded and $ u_{\alpha}$ is uniformly Lipschitz in $x$, as $\alpha\to 0$.
\end {proposition}
Using standard techniques we
can establish the following proposition, whose proof is presented in appendix \ref{Appendix3}:
\begin{proposition}\label{solviscalpha} 
Let
$u_{\alpha}  : \mathbb{R}^{n} \times \Omega \to \mathbb{R}$ be a solution of \eqref{halfie}. Then $u_{\alpha}$ is a viscosity solution (in $\omega$)
of  $\mathcal{H}(\varphi, 0, \omega)=0$, and $u_{\alpha}(0, \omega)$ is Lipschitz in $\omega$ with Lipschitz constant (in $\omega$)
 bounded by $K /\alpha$, where $K$ is independent
of $\alpha$, for all $\alpha\geq 0$.
\end {proposition}

\begin{proposition} \label{CritValueDiscounted}
Let  $u_{\alpha}$ be a viscosity solution in $\omega$ of
\eqref{halfie} Then
$$ \inf_{ \varphi \in C^{1}_{s}}  \sup_{ \omega \in \Omega} \left\lbrace 
- \alpha \int_{ \Omega} \varphi(0, \omega) d\nu (\omega) + \mathcal{H}^\alpha(\varphi,  0, \omega)\right\rbrace  = - \alpha \int_{ \Omega} u_{\alpha}(0, \omega) d\nu (\omega).$$
\end{proposition}
\begin{proof}
Consider a viscosity solution $u_{\alpha}$ of \eqref{halfie}. Then for any $\varphi \in C^{1}_{s}$ there exists a point $\omega_{ \varphi}$ of  minimum  for $u_{\alpha}(0, \omega ) - \varphi(0, \omega )$. Consider $ \varphi'(x, \omega )= \varphi(x, \omega ) + (u_{\alpha} - \varphi)(0, \omega_{ \varphi})$. Then $u_{\alpha}(0, \omega ) - \varphi'(0, \omega )$ has a minimum equal to 0 in $\omega_{ \varphi}$.

Since  $u_{\alpha}$  is a viscosity solution we have 
$ \mathcal{H}^\alpha(\varphi' ,  0, \omega_{ \varphi})  \geq 0 $ or equivalently
$$ \mathcal{H}^\alpha(\varphi  ,  0, \omega_{ \varphi}) + \alpha (u_{\alpha} - \varphi)(0, \omega_{ \varphi}) \geq 0.$$
Therefore
$$- \alpha \int_{ \Omega}  \varphi(0, \omega)  d\nu   + \mathcal{H}^\alpha(\varphi  ,  0, \omega_{ \varphi}) + \alpha
(u_{\alpha} - \varphi)(0, \omega_{ \varphi}) \geq  - \alpha \int_{ \Omega}  \varphi(0, \omega) d\nu,   $$
which implies
$$\sup_{ \omega \in \Omega} - \alpha \int_{ \Omega}  \varphi(0, \omega)  d\nu   + \mathcal{H}^\alpha(\varphi,  0, \omega)  \geq  - \alpha \int_{ \Omega}  \varphi(0, \omega)d\nu + \alpha (u_{\alpha} - \varphi)(0, \omega_{ \varphi}) ,  $$
and so
$$\sup_{ \omega \in \Omega} - \alpha \int_{ \Omega}  \varphi(0, \omega)  d\nu   + \mathcal{H}^\alpha(\varphi,  0, \omega)  \geq - \alpha \int_{ \Omega}  u_{\alpha}(0, \omega)  d\nu , $$
which finally yields
$$\inf_{ \varphi \in C^{1}_{s}}  \sup_{ \omega \in \Omega} \left\lbrace 
- \alpha \int_{ \Omega} \varphi(0, \omega) d\nu (\omega) + \mathcal{H}^\alpha(\varphi,  0, \omega)\right\rbrace  \geq - \alpha \int_{ \Omega}  u_{\alpha}(0, \omega)  d\nu. $$

In order to get the other inequality we use the functions $u^{\varepsilon}=u_\alpha  * \eta_\varepsilon$. Then $\mathcal{H}^\alpha(u^{\varepsilon}, 0, \omega) \leq o(1)$ owing to the convexity of the Hamiltonian and the uniform Lipschitz estimates on $u_{\alpha}$, we have 
$$\inf_{ \varphi \in C^{1}_{s}}  \sup_{ \omega \in \Omega} \left\lbrace 
- \alpha \int_{ \Omega} \varphi(0, \omega) d\nu (\omega) + \mathcal{H}^\alpha(\varphi,  0, \omega)\right\rbrace \leq  o(1) - \alpha \int_{ \Omega}  u^{\varepsilon}(0, \omega)  d\nu.$$
Then, the inequality desired is obtained by sending $\varepsilon$ to 0, and ends the proof. 
\end{proof}

\begin{corollary} \label{CritValueDiscLambda=0}
We have
$$\bar{H}_{\alpha} =\alpha \int_{ \Omega}  u_{\alpha}(0, \omega)  d\nu$$
where $u_{\alpha}$ is the unique viscosity solution of 
$ H(0, D_{x}u_{\alpha}(0, \omega), \omega) + \alpha u_{\alpha}(0, \omega)=0.$
\end{corollary}
\begin{proof}
In fact, if we apply Proposition \ref{CritValueDiscounted}  we have the formula
$$ \inf_{ \varphi \in C^{1}_{s}}  \sup_{ \omega \in \Omega} \left\lbrace 
- \alpha \int_{ \Omega} \varphi(0, \omega) d\nu (\omega) + \mathcal{H}^\alpha(\varphi,  0, \omega)\right\rbrace  = - \alpha \int_{ \Omega} u_{\alpha}(0, \omega) d\nu (\omega).$$
Remembering that 
$\bar{H}_{\alpha} = -  \inf_{ \varphi \in C^{1}_{s}}  \sup_{ \omega \in \Omega} 
- \alpha \int_{ \Omega} \varphi(0, \omega) d\nu (\omega) + \mathcal{H}^\alpha(\varphi,  0, \omega),$
we get
$$\bar{H}_{\alpha} =\alpha \int_{ \Omega}  u_{\alpha}(0, \omega)  d\nu.$$
\end{proof}

We state next, without proof, a partial converse to Proposition \ref{ViscSolEquiv}. The proof is rather technical and, in this paper, its only a application  is in Remark \ref{nonexistence}.

\begin{proposition}\label{ViscSolEquivConverse}
Suppose that,
\begin{itemize}
 \item[(a)] There exists $\delta >0$ such that, for all $x  \neq 0$ with $|x| < \delta$,
$\tau_{x}(\cdot): \Omega \to \Omega $  does not have fixed points.
\item[(b)] For each $\omega_{0} \in \Omega$, there exists $\delta >0$  and a set $\Sigma_{\delta}(\omega_{0}) \ni \omega_{0}$, such that, for all $x  \neq 0$ with $|x| < \delta$, and  $ \omega_{1}, \omega_{2} \in \Sigma_{\delta}(\omega_{0})$, if $\tau_{x}(\omega_{1})= \omega_{2} $  then $ \omega_{1} = \omega_{2}$. 
\item[(c)] The set 
\begin{equation}
\label{udelta}
\mathcal{U}_{\delta}(\omega_{0})=\left\lbrace \tau_{x}(\omega) \vert \; \omega \in \Sigma_{\delta}(\omega_{0}), \;  |x|<\delta/2 \right\rbrace 
\end{equation}
is an open neighborhood of $\omega_{0}$.
\end{itemize}
If   $u : \mathbb{R}^{n} \times \Omega \to \mathbb{R}$  is a viscosity solution in $\omega$ of $\mathcal{H}(u, 0, \omega) =\lambda$ then $u$ is also a viscosity solution in $x$  of  $\mathcal{H}^\alpha(u, 0, \omega)=\lambda$.
\end{proposition}

\begin{remark}\label{nonexistence}
Note that in some cases $\mathcal{H}^\alpha(u, \omega)=\lambda$ does not admit  viscosity solutions in $\omega$, as pointed out in \cite{LS}.
In their example $\Omega=\mathbb{T}^{2}$,   $L=L(x, v, \omega): \mathbb{R}  \times \mathbb{R} \times \mathbb{T}^{2} \to \mathbb{R}$  is the Lagrangian given by
$L(x, v, \omega)=\frac{1}{2} v^{2} + \cos(\omega_{1} +x) +  \cos(\omega_{2} + \sqrt{2}x),$
with the associated Hamiltonian 
$ H(x,p, \omega)=\frac{1}{2} p^{2} - \cos(\omega_{1} +x) -  \cos(\omega_{2} + \sqrt{2}x),$
and the action $\tau : \mathbb{R}  \times \mathbb{T}^{2}  \to \mathbb{T}^{2}$ is given by
$\tau_{x}(\omega_1, \omega_2)= (\omega_1 + x , \omega_2 + \sqrt{2} x).$ 

In this case the viscosity solutions in $x$ are unbounded.
So, if there where a viscosity solution in $\omega$,  then it would be a solution in $x$ by Proposition \ref{ViscSolEquivConverse}. By compactness, any stationary continuous function is bounded, which would be a contradiction.
\end{remark}

\section{Some formal computations} \label{sfc}
In this section we adapt the formal computations in \cite{EGom1} to motivate the regularity results in the following sections. 
Consider the periodic case of a $C^{2}$ Lagrangian $L:\mathbb{T} \times \mathbb{R} \to \mathbb{R}$, given by  $L(x, v)= \frac{1}{2} v^{2} - V(x)$, and the associated Hamiltonian $H(x,p)= \frac{1}{2} p^{2} + V(x)$. The stationary case follows along the same lines, as we will see in later sections.

Let $u$ be a solution to the discounted Hamilton-Jacobi equation $\frac{1}{2}u_{x}^{2} + V(x) + \alpha u=0$. Let $\mu_\alpha$ be a discounted Mather measure
with trace  $\theta_{\alpha}$ and such that the projection of $\mu_\alpha$ in the $x$ coordinated is denoted by $\theta$, that is, 
\[
\int_{\Tt\times \Rr} \varphi(x)d\mu_\alpha=\int_{\Tt} \varphi(x) d\theta.
\]
Note that $\theta$ in general does not agree with $\theta_\alpha$. In this section we  assume that $\mu_\alpha$ has the special property that 
$\theta_\alpha=\theta$. Under this assumption $\mu_\alpha$ is holonomic, that is
\[
\int_{\Tt\times \Rr} v \varphi_x(x) d\mu_\alpha=0, 
\]
for all $C^1$ periodic function $\varphi(x)$. 

We will first show that $\mu_\alpha$ almost every $(x,v)\in \Tt\times \Rr$, we have $v=-u_x(x)$. To see this we will argue by contradiction.
In this case if $v\neq -u_x(x)$, there would exist a set of positive measure $\mu_\alpha$ in which 
\[
L(x,v)+v u_x > -H(u_x, x). 
\]
Since $L(x,v)+v u_x \geq -H(u_x, x)$, integrating with respect to $\mu_\alpha$ yields
\[
\int_{\Tt\times \Rr} L d \mu_\alpha+ \int_{\Tt\times \Rr} v u_x d\mu_\alpha>\alpha \int_{\Tt\times \Rr} u d\mu_\alpha.
\]
This would yield
\[
\int_{\Tt\times \Rr} L d \mu_\alpha>\alpha \int_{\Tt} u d\theta_\alpha,
\]
which contradicts the optimality condition.

Therefore the holonomy constraint can be written as
\[
\int_{\mathbb{R}} ( u_{x} \varphi_{x} + \alpha \varphi ) d\theta(x)= \alpha \int_{\mathbb{R}} \varphi d\theta_{\alpha}(x).
\]

By differentiating twice the Hamilton-Jacobi equation we have $ u_{x}(u_{xx})_{x} +u_{xx}^{2}  + V''(x) + \alpha u_{xx}=0$. Integrating with respect to  $\mu_{\alpha}$ yields
$$\int_{\mathbb{R}} ( u_{x}(u_{xx})_{x} +u_{xx}^{2}  + V''(x) + \alpha u_{xx} ) d\mu_\alpha=0,$$
or, equivalently,
$$\int_{\mathbb{R}} u_{x}(u_{xx})_{x} d\theta(x)  + \int_{\mathbb{R}} u_{xx}^{2} + \alpha u_{xx}   d\theta(x)= - \int_{\mathbb{R}} V''(x)   d\theta(x). $$
Since the trace of $\mu_{\alpha}$, $\theta_\alpha$ is equal to its projection $\theta(x)$, then the measure $\mu_{\alpha}$ is holonomic and so $$\int_{\mathbb{R}} u_{x}(u_{xx})_{x} d\theta(x)=0.$$ 
Using   $-  \alpha u_{xx}   \leq  \frac{1}{2} u_{xx}^{2} +  \frac{1}{2}\alpha^{2}$ we get,
\begin{align*}
\int_{\mathbb{R}} u_{xx}^{2} d\theta(x)&= - \int_{\mathbb{R}}  V''(x)   d\theta(x) -  \int_{\mathbb{R}}  \alpha u_{xx}   d\theta(x)  \\& \leq    - \int_{\mathbb{R}}  V''(x)   d\theta(x) +  \int_{\mathbb{R}}  \frac{1}{2} u_{xx}^{2} +  \frac{1}{2}\alpha^{2}   d\theta(x),
\end{align*}
which  yields a $L^{2}(\theta)$ bound for $u_{xx}$:
$$\int_{\mathbb{R}} u_{xx}^{2} d\theta(x) \leq    \int_{\mathbb{R}}  \alpha^{2} - 2 V''(x) d\theta(x).$$
In order to derive $L^{\infty}$ estimates to $u_{xx}$ we proceed as follows: first we multiply the second derivative of the Hamilton-Jacobi equation by a function $\Psi'(u_{xx})$,  
$$\int_{\mathbb{R}} u_{x}(u_{xx})_{x}\Psi'(u_{xx}) d\theta(x)  + \int_{\mathbb{R}} \left[ u_{xx}^{2} + \alpha u_{xx} + V''(x) \right ]   \Psi'(u_{xx}) d\theta(x)=0.$$
Let $\Psi: \mathbb{R} \to \mathbb{R}$ be such that
$$\Psi'(x)=
\begin{cases}
1  & \text{if} \; x \leq -\lambda  \\
0  & \text{otherwise},
\end{cases}
$$
where $\lambda >0$ is fixed. Choose $\Phi(x)=\Psi'(x)$ (actually one should to use a $C^{\infty}$ approximation of $\Psi'(x)$). Observe that $(u_{xx})_{x}\Phi(u_{xx})=\Psi(u_{xx})_{x}$ and so $\int_{\mathbb{R}}u_{x}(u_{xx})_{x}\Psi'(u_{xx})   d\theta(x) =0$.
Define $A=\{x | u_{xx} \leq -\lambda \}$. Thus, 
\begin{align*}
0&= \int_{\mathbb{R}} \left(u_{xx}^{2}  + V''(x) +\alpha u_{xx}\right) \Phi(u_{xx}) d\theta(x)
\\&=\int_{A} \left(u_{xx}^{2}  + V''(x) +\alpha u_{xx}\right) d\theta(x).
\end{align*}
Since $u_{xx} \leq -\lambda$, and using $  \alpha u_{xx}   \leq  -\frac{1}{2} u_{xx}^{2} -   \frac{1}{2}\alpha^{2}$, one can show that, $0 \geq (\frac{\lambda^2}{2} -   \frac{1}{2}\alpha^{2} + c) \theta(A)$, where, $|V''|  \leq c $. Since $\lambda$ is arbitrary , we get $\theta(A)=0$.  Thus, there exists $\lambda >0$, such that, $u_{xx} > -\lambda$, $\theta$-a.e.

The solutions of  $\alpha u+\frac{1}{2}u_{x}^{2} + V(x)=0$ are semi-concave (this is a standard result, see \cite{Bardi} or the survey paper \cite{BG}), so we get that there exists $\beta >0$ such that $u_{xx} <  \beta$, and so, for some $C>0$,
$|u_{xx}| < C$, $\theta$ almost everywhere.

\section{Holonomic discounted stationary Mather measures}
\label{hdmm}

Motivated by the formal computations in the previous section, we will now establish the existence of holonomic discounted stationary Mather measures. In the paper \cite{CamCapGom}, these measures were called invariant, we did not keep this name here to avoid confusion with invariance with respect to Euler- Lagrange equations.

Given a probability measure $\nu$, and a corresponding discounted stationary  Mather measure $\mu$ with trace $\nu$,  we say that $\mu$ is a holonomic discounted stationary Mather measure if
$$\int_{\mathbb{R}^{n} \times \Omega}  \varphi (0, \omega)  d\mu (v, \omega)= \int_{ \Omega} \varphi(0, \omega) d\nu (\omega),$$
for all $\varphi \in C^{1}_{s}(\mathbb{R}^{n} \times \Omega)$. In particular, $\mu$ satisfies the undiscounted holonomy constraint.
\begin{theorem}\label{InvarExistence} 
There exists a holonomic discounted stationary Mather measure. 
\end {theorem}
\begin{proof}
Fix $\omega\in \Omega$. Consider a sequence $T_n\to \infty$ and a sequence $x_n(t)$ 
of minimizing  trajectories for the dynamic 
programing principle \eqref{DynProgPrinciple}, that is, 
$$u_{\alpha}(0, \omega)=   \int_{0}^{T_n} e^{- \alpha t} L(x_n(t), \dot{x}_n(t), \omega) dt + 
e^{- \alpha T_n} u_{\alpha}(x_n(T_n), \omega). $$
Because $u_\alpha$ is Lipschitz and 
\[
\dot x_n=-D_pH(D_xu^\alpha(x_n(t)), x_n(t))
\]
the $|\dot x_n|$ is uniformly bounded. 

Define a probability measure $\mu$ by
$$\int_{\mathbb{R}^{n} \times \Omega}  \phi (v, \omega)  d\mu (v, \omega)=\lim_{n \to \infty} \frac{1}{T_n} \int_{0}^{T_n} \phi(\dot{x}_n, \tau_{x_n(t)}\omega)  dt,$$
for any $\phi \in C_\gamma^0(\mathbb{R}^{n} \times \Omega)$, where the limit is taken through an appropriate subsequence. This sublimit exists and is a probability measure
because $\Omega$ is compact
and $|\dot x_n|$ is uniformly bounded. 

Let $\varphi \in C_s^1$. Observe that $\frac{d \,}{dt}\varphi (x_n(t), \omega)= \dot{x}_n(t) \cdot D_{x} \varphi (0,\tau_{x_n(t)} \omega)$. So, if 
$\phi (v, \omega)=v \cdot D_{x} \varphi (0, \omega)$, then 
\begin{align*}
&\int_{\mathbb{R}^{n} \times \Omega}  \phi (v, \omega)  d\mu (v, \omega)=\lim_{n \to \infty}  \frac{1}{T_n} \int_{0}^{T_n} \dot{x}_n(t)  \cdot D_{x} \varphi (0,\tau_{x_n(t)} \omega)  dt\\  &= \lim_{n \to \infty}  \frac{1}{T_n} \int_{0}^{T_n} \frac{d \,}{dt}\varphi (x_n(t), \omega)  dt  =
\lim_{n \to \infty}\frac{\varphi(x_n(T_n))-\varphi(x_n(0))}{T_n}=
0. 
\end{align*}
Since $A^{v} \varphi  =v \cdot D_{x}\varphi(0, \omega) - \alpha \varphi(0, \omega)$,
\begin{align*}
&\int_{\mathbb{R}^{n} \times \Omega} A^{v} \varphi   d\mu = \int_{\mathbb{R}^{n} \times \Omega}  v \cdot  D_{x} \varphi (0,  \omega) - \alpha \varphi (0,  \omega)  d\mu (v, \omega)\\&=\lim_{n \to \infty}  \frac{1}{T_n} \int_{0}^{T_n} \dot{x}_n(t)  \cdot D_{x} \varphi (0,\tau_{x_n(t)} \omega) - \alpha \varphi (0, \tau_{x_n(t)} \omega)  dt  \\&=- \alpha  \lim_{T_n \to \infty}  \frac{1}{T_n} \int_{0}^{T_n}  \varphi (0, \tau_{x_n(t)} \omega)  dt  =- \alpha  \int_{ \Omega} \varphi(0, \omega) d\nu (\omega),    
\end{align*}
where $\nu$ is given by,
$$\int_{ \Omega} g(\omega) d\nu(\omega)= \lim_{n \to \infty}  \frac{1}{T_n} \int_{0}^{T_n}  g(\tau_{x_n(t)} \omega)  dt,$$
for all continuous function $g : \Omega \to \mathbb{R}$. In particular, $\int_{\mathbb{R}^{n} \times \Omega}  \varphi (0, \omega)  d\mu (v, \omega)= \int_{ \Omega} \varphi(0, \omega) d\nu (\omega)$.

We must to prove that $\mu$ is minimizing. To do so, fix first $n$ and
consider a partition $\{0=t_{0}, t_{1},..., t_{N-1}=T_n \}$ of $[0,T_n]$,  where $t_{i+1}=t_{i} + h$, and $h=T_n/N$.
 The restriction of $x_n(t)$ to each sub-interval is   minimizing, i.e., 
$$u_{\alpha}(x_n(t_{i}), \omega)=   \int_{t_{i}}^{t_{i+1}} e^{- \alpha (t- t_{i})} L(x_n(t), \dot{x}_n(t), \omega) dt + 
e^{- \alpha h} u_{\alpha}(x_n(t_{i+1}), \omega).$$
We have,
$$\sum_{i=0}^{i=N-1} u_{\alpha}(x_n(t_{i}), \omega) - e^{- \alpha h} u_{\alpha}(x_n(t_{i+1}), \omega)=  $$
$$=\sum_{i=0}^{i=N-1} u_{\alpha}(x_n(t_{i}), \omega) - u_{\alpha}(x_n(t_{i+1}), \omega) + (1- e^{- \alpha h}) u_{\alpha}(x_n(t_{i+1}), \omega)= $$ 
$$= u_{\alpha}(x_n(0), \omega) - u_{\alpha}(x_n(T), \omega) +\alpha (\frac{1- e^{- \alpha h}}{\alpha h}) \sum_{i=0}^{i=N-1} h u_{\alpha}(x_n(t_{i+1}), \omega).$$
Sending $h \to 0$ we get 
$$\lim_{h \to 0} \sum_{i=0}^{i=N-1} u_{\alpha}(x_n(t_{i}), \omega) - e^{- \alpha h} u_{\alpha}(x_n(t_{i+1}), \omega)=  u_{\alpha}(x_n(0), \omega) - u_{\alpha}(x_n(T_n), \omega) +\alpha \int_{0}^{T_n}  u_{\alpha}(0, \tau_{x_n(t)} \omega)dt. $$ 
On the other hand, we have
$$ \lim_{h \to 0} \sum_{i=0}^{i=N-1} \int_{t_{i}}^{t_{i+1}} e^{- \alpha (t- t_{i})} L(x_n(t), \dot{x}_n(t), \omega) dt = \int_{0}^{T_n}  L(x_n( t ), \dot{x}_n( t ), \omega) dt.$$
Thus, 
\begin{align*}
&\alpha \int  u_{\alpha}(0, \omega)  d\nu=\\
&\lim_{n \to \infty} \frac{1}{T_n} \left\lbrace u_{\alpha}(x_n(0), \omega) - u_{\alpha}(x_n(T_n), \omega) +\alpha \int_{0}^{T_n}  u_{\alpha}(0, \tau_{x_n(t)} \omega)dt \right\rbrace  \\
&=\lim_{n \to \infty} \frac{1}{T_n} \int_{0}^{T_n}  L(x_n(t), \dot{x}_n(t), \omega) dt=\int_{\mathbb{R}^{n} \times \Omega} L (0, v, \omega)  d\mu (v, \omega).
\end{align*}
%We conclude then that  
%$$ \alpha \int  u_{\alpha}(0, \omega)  d\nu = \int_{\mathbb{R}^{n} \times \Omega} L (0, v, \omega)  d\mu (v, \omega).$$
By Corollary \ref{CritValueDiscLambda=0} we have $\bar{H}_{\alpha} =\alpha \int  u_{\alpha}(0, \omega)  d\nu$.
Thus $\mu$ is minimizing.
\end{proof}

We should note here that the Theorem does not assert uniqueness. Furthermore the measure $\mu$ may depend on the choice of $\omega\in \Omega$ or in the particular
sequence we choose to extract the weak limits. For our purposes, however, existence is sufficient. 

\begin{theorem}
\label{discoutedelinvar}
Let $\mu_\alpha$ be a holonomic discounted Mather measure as constructed in theorem \ref{InvarExistence}. Then $\mu_\alpha$ is invariant under the discounted Euler-Lagrange flow. 
\end{theorem}
\begin{proof}
It suffices to prove that for any bounded function $\phi(x,v,\omega)\in C^1_s(\Rr^n\times \Rr^n\times \omega)$ we have
\[
\int_{\Rr^n\times \Omega} W^{L_\alpha} \nabla_{x,v} \phi(0, v, \omega) d\mu_\alpha=0.
\]
This follows, from the identity
\begin{align*}
\phi(x_n(T_n), \dot x_n(T_n), \omega)-\phi(x_n(0), \dot x_n(0), \omega) &=\int_0^{T_n} \frac{d}{dt}\phi(x_n(t), \dot z_n(t), \omega)\\
&=\int_0^{T_n} X^{L_\alpha}\frac{\partial \phi}{\partial x} + Y^{L_\alpha} \frac{\partial \phi}{\partial v}, 
\end{align*}
dividing by $T_n$ and letting $n\to \infty$.
\end{proof}

\section{Graph property, regularity and stationary Mather measures}
\label{graph}

In this section establish that the discounted Mather measures are supported in a graph of a (partially) Lipschitz function.
As we are using similar techniques to \cite{EGom1} (see also \cite{BG}) we will present in this section the main differences and technical points and postpone to Appendix \ref{apC}
the detailed proofs. We will the discounted Mather measures to construct a stationary Mather measure invariant under the Euler-Lagrange flow. 

We assume that
\[
L(x+y, v, \omega) -L(x, v, \omega) \leq (c +c L) |y|
% \leq \frac{C}{\alpha} |y|
\]
\begin{lemma}\label{EstabilityAlphaViscSolutions}  
Let $u_{\alpha}$ be the viscosity solution of  $H(0, D_{x}u_{\alpha}(0, \omega), \omega) + \alpha u_{\alpha}(0, \omega)=0 $ given by Proposition \ref{solviscalpha}. 
Then $$ \lim_{\alpha \to 0}  \alpha  u_{\alpha}(0, \omega)$$
does not depend on $\omega$.
%  = \bar{H},$$
%where  $\displaystyle \bar{H}=\min \int_{\mathbb{R}^{n} \times \Omega} L(0, v, \omega) d\mu (v, \omega)$.
\end{lemma}
\begin{proof}
We know that $\alpha u_{\alpha}$ is uniformly bounded, so $\alpha u_{\alpha}(0, \omega) \to \xi( \omega)$ pointwise for some function. On the other hand, fixed $\omega_{0} \in \Omega$ we know that $u_{\alpha}(y, \omega_{0})$ is uniformly Lipschitz in $x$, uniformly as $\alpha\to 0$, that is,
$$| u_{\alpha}(x_{1}, \omega_{0}) - u_{\alpha}(x_{2}, \omega_{0}) | < C |x_{1} - x_{2} |.$$
Thus, if $|y| < R$ then,
$$\lim_{\alpha \to 0} | \alpha u_{\alpha}(y, \omega_{0}) -\alpha u_{\alpha}(0, \omega_{0}) | < \lim_{\alpha \to 0}\alpha C |y | = 0,$$
that is, $\lim_{\alpha \to 0} \alpha u_{\alpha}(0, \tau_{y}\omega_{0}) = \lim_{\alpha \to 0} \alpha u_{\alpha}(0, \omega_{0}) $ for $|y| < R$.

From Proposition \ref{solviscalpha} we know that $u_{\alpha}(0, \omega)$ is Lipschitz in $\omega$ with Lipschitz constant $K /\alpha$, that is,
$$| u_{\alpha}(0, \omega_{1}) - u_{\alpha}(0, \omega_{2}) | < \frac{K}{\alpha}  d(\omega_{1} , \omega_{2}).$$
Consider $\varepsilon >0$ and $y \in \mathbb{R}^{n}, \text{ such that } d(\tau_{y} \omega_{0},  \omega_{1}) < \varepsilon$. Observe that,
$$| \alpha u_{\alpha}(0, \omega_{0}) -\alpha u_{\alpha}(0, \omega_{1}) | \leq $$
$$ \leq | \alpha u_{\alpha}(0, \omega_{0}) - \alpha u_{\alpha}(0, \tau_{y} \omega_{0})  | +  |\alpha u_{\alpha}(0, \tau_{y} \omega_{0}) - \alpha u_{\alpha}(0, \omega_{1}) | \leq $$
$$ \leq | \alpha u_{\alpha}(0, \omega_{0}) - \alpha u_{\alpha}(0, \tau_{y} \omega_{0})  | +\alpha \frac{K}{\alpha} d(\tau_{y} \omega_{0},  \omega_{1}).$$
Sending $\alpha \to 0$, and then $\varepsilon\to 0$ we get $\lim_{\alpha \to 0}\alpha u_{\alpha}(0, \omega_{0}) =\lim_{\alpha \to 0} \alpha u_{\alpha}(0, \omega_{1})$. Thus, $\xi( \omega)$ is constant.
%It is easy to see that then $\xi(\omega)=\bar{H}$.
\end{proof}

\begin{lemma}
Let
$u_{\alpha}$ be the viscosity solution of  $H(0, D_{x}u_{\alpha}(0, \omega), \omega) + \alpha u_{\alpha}(0, \omega)=0 $ given by Proposition \ref{solviscalpha}. 
Then
$$ \lim_{\alpha \to 0}  \alpha  u_{\alpha}(0, \omega)=\Hh,$$
where 
\[
\Hh=\inf \int_{\Rr^n\times \Omega} L d\mu, 
\]
and the infimum is taken over all stationary holonomic measures.
\end{lemma}
\begin{proof}
Denote by $\tilde H$ the limit as $\alpha\to 0$ of $\alpha u_\alpha$, which is constant by the previous lemma.   
Let $\mu_\alpha$ be a holonomic discounted stationary Mather measure. 
Then, because $\mu_\alpha$ is holonomic we have
\[
\Hh\leq \lim_{\alpha\to 0} \int_{\Rr^n\times \Omega} L d\mu_\alpha =\lim_{\alpha\to 0} \alpha \int_{\Rr^n\times \Omega} u_\alpha d\mu_\alpha=\tilde H.
\]
Let $\mu$ be a stationary Mather measure. 
Then, because $\mu$ is a discounted holonomic measure with trace $\mu$ we have
\[
\Hh=\int_{\Rr^n\times \Omega} L d\mu\geq \alpha \int_{\Rr^n\times \Omega} u_\alpha d\mu \to \tilde H, 
\]
as $\alpha\to 0$. This shows that $\tilde H=\Hh$.
\end{proof}

\begin{lemma}\label{EstabilityAlphaMeasures}  
Let $\mu_{\alpha}$ be a sequence of discounted stationary Mather measures with trace $\nu_{\alpha}$. Suppose that $\mu_{\alpha} \to \mu$ when $ \alpha \to 0 $, then $\mu$ is a stationary
Mather measure.
\end {lemma}
\begin{proof}
First we must to prove that $\mu$ is a holonomic probability measure. In fact, for any $\varphi\in C^1_s$,
$$\int_{\mathbb{R}^{n} \times \Omega} v \cdot D_{x}\varphi(0, \omega) d\mu_{\alpha} = \alpha \int_{\mathbb{R}^{n} \times \Omega} \varphi(0, \omega) d\mu_{\alpha} - \alpha  \int_{ \Omega} \varphi(0, \omega) d\nu_{\alpha}  \to 0,$$
when $\alpha \to 0$.

Using Corollary \ref{CritValueDiscLambda=0} we get 
$$\int_{\mathbb{R}^{n} \times \Omega} L d\mu= \lim_{\alpha \to 0} \int_{\mathbb{R}^{n} \times \Omega} L d\mu_{\alpha}= \lim_{\alpha \to 0} \bar{H}_{\alpha}= \lim_{\alpha \to 0} \int_{ \Omega} \alpha  u_{\alpha}(0, \omega)  d\nu_{\alpha}(\omega) = \bar{H}.$$  
Thus $\mu$ is a Mather measure. 
\end{proof}

\begin{theorem}\label{GraphPropMatherMeasures}  
Let $\mu_{\alpha}$ be a discounted Mather measure with trace $\nu_{\alpha}$ (or if $\alpha=0$ a stationary Mather measure).
Then  $\mu_{\alpha}$ is supported in a graph, that is, there exists a measurable function $V_{\alpha}: \Omega \to  \mathbb{R}^{n}$ such that,
$$\supp \mu_{\alpha} = \{ (v, \omega) \in \mathbb{R}^{n} \times \Omega |  v=V_{\alpha}(\omega) \}.$$
\end {theorem}
\begin{proof}
As in \cite{BG}, for instance, we just observe that the result follows from the fact that the Lagrangian is strictly convex in $v$, whereas the discounted holonomy
constraint is linear in $v$. 
\end{proof}

Since the holonomic discounted measures are also holonomic, the same techniques in \cite{EGom1} (see also \cite{BG}) can be adapted to establish the following regularity result:
\begin{theorem}\label{LipschitzOfViscSol}
Let $\mu_{\alpha}$ be a holonomic discounted Mather measure.  If $u_{\alpha}$ is a viscosity solution of \eqref{DiscHJE}, then for each $y \in \mathbb{R}$, 
$$| D_{x}u_{\alpha}(y, \omega)   -  D_{x}u_{\alpha}(0, \omega) | \leq  C  | y |,$$
$\theta$ almost everywhere and uniformly in $\alpha$.
\end {theorem}

The proof of this theorem since it follows (almost) exactly the same steps as in \cite{EGom1} (see also \cite{BG}) and is presented for completeness
in appendix \ref{apC}.
 The only difference
is the term $\alpha u$ in the Hamilton-Jacobi equation, which can be controlled, as discussed in section \ref{sfc}, because we are using holonomic
discounted measures. As a corollaries to the previous Theorem we have

\begin{corollary}\label{ThGraphLipsLimitMeasure}  
Let $\mu_{\alpha}$ be a holonomic discounted Mather measure. Then, there exists a function $V_\alpha: \Omega \to  \mathbb{R}^{n}$, such that
$\supp \mu_\alpha = \{ (v, \omega) \in \mathbb{R}^{n} \times \Omega |  v=V_{\alpha}(\omega) \}$.
Furthermore, $V_\alpha$ is partially Lipschitz in the following sense:
\[
|V_\alpha(\tau_y\omega)-V_\alpha(\omega)|\leq C|y|, 
\]
for all $\omega$ in the support of $\mu_\alpha$, and
$C$ is uniformly bounded as $\alpha\to 0$. 
\end {corollary}
%\noindent and
%\begin{corollary}
%Let $\mu_{\alpha}$ be a sequence of holonomic discounted stationary Mather measures. Suppose $\mu_\alpha\to \mu$. Then $\mu$ is a stationary Mather measure, and
%there exists a Lipschitz function 
%$V: \Omega \to  \mathbb{R}^{n}$, such that
%$\supp \mu = \{ (v, \omega) \in \mathbb{R}^{n} \times \Omega |  v=V(\omega) \}$.
%\end{corollary}

%\section{Stationary Mather measures}
%\label{invar}

Finally, our last result concerns the existence of stationary Mather measures invariant under the Euler-Lagrange flow. 

\begin{theorem}
\label{InvarianceFlowMeasures}
There exists a stationary Mather measure $\mu$ which is invariant under the Euler-Lagrange flow. Furthermore $\mu$ is supported on a graph.
\end{theorem}
\begin{proof}
Let $\mu_\alpha$ be holonomic discounted Mather measures as constructed in theorem \ref{InvarExistence}. Consider a weak limit $\mu$. 
By lemma \ref{EstabilityAlphaMeasures}, $\mu$ is a stationary Mather measure. Because for any $\phi(x,v)$ we have
\[
\int_{\Rr^n\times \Omega} W^{L_0} \nabla_{x,v} \phi(0, v, \omega) d\mu_\alpha=\alpha \int_{\Rr^n\times \Omega} (D^2_{vv} L)^{-1} D_v L D_v\phi(0, v, \omega) d\mu_\alpha.
\]
we conclude that
\[
\int_{\Rr^n\times \Omega} W^{L_0} \nabla_{x,v} \phi(0, v, \omega) d\mu=0. 
\]
The graph property of stationary Mather measures follows from theorem \ref{GraphPropMatherMeasures}. 
\end{proof}

\appendix

\section{Proof of Theorem \ref{Duality}}
\label{apA}

In this appendix we present the proof of Theorem \ref{Duality}, as well as some background material.

Let $\gamma$ be as in \eqref {gammadef}.
Let $\Mm$ be the set of weighted Radon measures on $\Omega \times \mathbb{R}^{n}$, i.e.,
$$\Mm=\{\text{signed measures on }\mathbb{R}^{n} \times \Omega \text{ with }
\int_{\mathbb{R}^{n} \times \Omega} \gamma d|\mu|<\infty\}.$$
Note that $\Mm$ is the dual of
the set $C_\gamma^0(\mathbb{R}^{n} \times \Omega)$.

We introduce the following sets
$$\Mm_{1}=\{ \mu \in \Mm \vert   \; \mu  \;  \text{is a positive probability measure} \},$$
and
$$\Mm_{2}=\{ \mu \in \Mm \vert  \int_{\mathbb{R}^{n} \times \Omega} v \cdot D_{x}\varphi(0, \omega) d\mu (v, \omega)=0, \;
\text{for all } \varphi \in C^{1}_{s}(\mathbb{R}^{n} \times \Omega) \}.$$

Using this notation the Mather problem can be reformulated as
$$\displaystyle \min_{\Mm_{1} \cap \Mm_{2}}  \int_{\mathbb{R}^{n} \times \Omega} L(0, v, \omega) d\mu (v, \omega).$$

Consider the following subset of  functions $\phi : \mathbb{R}^{n} \times \Omega \to \mathbb{R}$, 
$$\mathcal{C} =\cl  \{\phi \; \vert \; \phi(v, \omega)=
v \cdot D_{x}\varphi(0, \omega),  \text{ for some }  \varphi \in C^{1}_{s}(\mathbb{R}^{n} \times \Omega) \}.$$
Observe that $ \mathcal{C}$ is a  closed convex set. 

For $\phi \in C_\gamma^0$ let
\begin{equation}
\label{formulah}
h(\phi)=\sup_{ \mathbb{R}^{n} \times \Omega} (  -\phi(v, \omega) - L(0, v, \omega)).
\end{equation}
Since $h$ is the supremum of linear functions, it is
a convex function on  $C_\gamma^0$. As we will see bellow in Lemma \ref{continuity}, $h$ is a continuous function.  

For $\phi \in C_\gamma^0$, let
\begin{equation}
\label{formulag}
g(\phi)=
\begin{cases}
0       \quad          &\text{if}               \quad \phi  \in  \mathcal{C} \\
-\infty                  &\text{otherwise.}
\end{cases}
\end{equation}

As $\mathcal{C}$ is a closed convex set we have that $g$ is a concave and upper semicontinuous function. Therefore
its Legendre-Fenchel transform is given by
\begin{equation}
\label{lftg2}
g^*(\mu)=\inf_{\phi \in C_\gamma^0(\mathbb{R}^{n} \times \Omega)}  \left ( - \int_{\mathbb{R}^{n} \times \Omega} \phi d\mu  - g(\phi) \right ).
\end{equation}
Since $h$ is a convex and lower semicontinuous function, its Legendre-Fenchel transform is given by
\begin{equation}
\label{lfth}
h^*(\mu)=\sup_{\phi \in C_\gamma^0(\mathbb{R}^{n} \times \Omega)}  \left ( -\int_{\mathbb{R}^{n} \times \Omega} \phi d\mu  - h(\phi) \right ).
\end{equation}

\begin{proposition}\label{legendre}
Let $g$ and $h$ defined as in \eqref{formulah} and \eqref{formulag}. Then
$$
h^*(\mu)=
\begin{cases}
\int_{\mathbb{R}^{n} \times \Omega} L(0, v, \omega) d\mu(v, \omega) \quad &\text{if}               \quad \mu \in \Mm_{1} \\
+\infty                  &\text{otherwise,}
\end{cases}
$$
and
$$
g^*(\mu)=
\begin{cases}
0 \quad & \text{if} \quad \mu \in \Mm_{2} \\
- \infty  & \text{otherwise.}
\end{cases}
$$
\end{proposition}
\begin{proof}
First we assume that $\mu \in \Mm_{1}$.
  As $h$ is a convex function, its Legendre transform is given by \eqref{lfth}.
Using the definition of $h$, equation \eqref{formulah},
 we get
$$h^*(\mu)=\sup_{\phi \in C_\gamma^0(\mathbb{R}^{n} \times \Omega)} \left ( -\int_{\mathbb{R}^{n} \times \Omega} \phi d\mu  - \sup_{ \mathbb{R}^{n} \times \Omega} ( - \phi(v, \omega) - L(0, v, \omega)) \right ).$$

Consider the family of compact subsets of $\mathbb{R}^{n} \times \Omega$ given by
$$K_{n}=\{(v, \omega) \in  \mathbb{R}^{n} \times \Omega\,  |\,  |v| \leq n \},$$
and let $\eta_{n} : \mathbb{R}^{n} \times \Omega \to \mathbb{R}$ be a continuous function such that $0 \leq \eta_{n} \leq 1 $, $\eta_{n} =1$ in $K_{n}$, and $\supp \eta_{n} \subset  K_{n+1}$.
Then define 
$$L_{n}=L(0, v, \omega) \cdot \eta_{n}(v, \omega).$$
Observe that the sequence $L_{n}$ is increasing and pointwise convergent to $L(0, v, \omega)$. 

Is easy to see that $L_{n} \in C_\gamma^0(\mathbb{R}^{n} \times \Omega)$. Furthermore, for fixed $n$, one can write any function
$\phi \in C_\gamma^0(\mathbb{R}^{n} \times \Omega)$   as $\phi=-L_{n} + \psi$ where $\psi \in C_\gamma^0(\mathbb{R}^{n} \times \Omega)$. From this observation we get
\begin{align}
\notag h^*(\mu)&=\sup_{\psi \in C_\gamma^0(\mathbb{R}^{n} \times \Omega)} \left ( -\int_{\mathbb{R}^{n} \times \Omega} (-L_{n} + \psi) d\mu  - \sup_{ \mathbb{R}^{n} \times \Omega} ( -(-L_{n} + \psi) - L) \right )\\ \notag &
=\sup_{\psi \in C_\gamma^0(\mathbb{R}^{n} \times \Omega)} \left ( \int_{\mathbb{R}^{n} \times \Omega} L_{n}  d\mu  - \int_{\mathbb{R}^{n} \times \Omega} \psi d\mu - \sup_{ \mathbb{R}^{n} \times \Omega} ( L_{n} - L -  \psi ) \right )
\\&= \int_{\mathbb{R}^{n} \times \Omega} L_{n}  d\mu  + \sup_{\psi \in C_\gamma^0(\mathbb{R}^{n} \times \Omega)}  \int_{\mathbb{R}^{n} \times \Omega} \left (-\psi  - \sup_{ \mathbb{R}^{n} \times \Omega} ( L_{n} - L -  \psi ) \right ) d\mu. 
\label{(1)}     
\end{align}
If we take $\psi =0$ in \eqref{(1)} we have
$$h^*(\mu) \geq  \int_{\mathbb{R}^{n} \times \Omega} L_{n}  d\mu  + \int_{\mathbb{R}^{n} \times \Omega} \left ( - \sup_{ \mathbb{R}^{n} \times \Omega} ( L_{n} - L ) \right ) d\mu  \geq \int_{\mathbb{R}^{n} \times \Omega} L_{n}  d\mu.$$
Thus using the monotone convergence theorem we get
$$h^*(\mu) \geq  \int_{\mathbb{R}^{n} \times \Omega} L d\mu.$$
In order to get the other inequality we can rewrite \eqref{(1)} as follows
\begin{align*}
h^*(\mu)&= \int_{\mathbb{R}^{n} \times \Omega} L_{n}  d\mu  + \sup_{\psi \in C_\gamma^0}  
\int_{\mathbb{R}^{n} \times \Omega} \left ((S  -\psi -S )  - \sup_{ \mathbb{R}^{n} \times \Omega} ( S -  \psi ) \right ) d\mu \\&= \int_{\mathbb{R}^{n} \times \Omega} L_{n}  d\mu  + \sup_{\psi \in C_\gamma^0}  
\int_{\mathbb{R}^{n} \times \Omega} \left ( (S -\psi)  - \sup_{ \mathbb{R}^{n} \times \Omega} ( S -  \psi ) \right ) d\mu  - \int_{\mathbb{R}^{n} \times \Omega}  S d\mu, 
\end{align*}
where $S=L_{n} - L$.
Since $\mu \in \Mm_{1}$,  we have $\int_{\mathbb{R}^{n} \times \Omega} \left( (S -\psi)  - \sup_{ \mathbb{R}^{n} \times \Omega} ( S -  \psi ) \right ) d\mu \leq 0$. Therefore
$$h^*(\mu) \leq \int_{\mathbb{R}^{n} \times \Omega} L_{n}  d\mu  - \int_{\mathbb{R}^{n} \times \Omega}  S d\mu = \int_{\mathbb{R}^{n} \times \Omega} L d\mu.$$

If  $\mu \not\in \Mm_{1}$, we have two possibilities.
First, if  $\mu \not\geq 0$ then we can find a positive function $\psi \in C_\gamma^0(\mathbb{R}^{n} \times \Omega)$  such that $\int \psi d\mu < 0 $. Define $\psi_{n}=   n \psi \in C_\gamma^0(\mathbb{R}^{n} \times \Omega)$, then 
\begin{align*}
h^*(\mu) &\geq \left (- \int_{\mathbb{R}^{n} \times \Omega} \psi_{n} d\mu  - \sup_{ \mathbb{R}^{n} \times \Omega} ( - \psi_{n} - L) \right ) \\ &=
n \left ( \int_{\mathbb{R}^{n} \times \Omega} -\psi d\mu  + \inf_{ \mathbb{R}^{n} \times \Omega} ( \psi + \frac{1}{n}L) \right )  \to  + \infty, 
\end{align*}
when $n \to \infty$.

On the other hand, if $\mu \geq 0$ but $\int d\mu \neq 1$ we take $\phi = k \in  \mathbb{R}$, then
\begin{align*}
h^*(\mu) &\geq \left (- \int_{\mathbb{R}^{n} \times \Omega} k d\mu  - \sup_{ \mathbb{R}^{n} \times \Omega} ( - k - L) \right ) \\&
= k  \left (1  -  \int_{\mathbb{R}^{n} \times \Omega}  d\mu \right ) +\inf_{ \mathbb{R}^{n} \times \Omega}  L \to  + \infty
\end{align*}
when $k \to \pm \infty$, because  $L \geq 0$.

Now we compute the Legendre transform of $g$. As $g$ is concave we compute its Legendre-Fenchel transform using \eqref{lftg2}.
First we suppose $ \mu \in \Mm_{2}$. In this case we have two possibilities, if $\phi  \in  \mathcal{C}$ then
$$-\int_{\mathbb{R}^{n} \times \Omega} \phi d\mu  - g(\phi)= 0,$$
else if, $\phi  \not\in  \mathcal{C}$ then
$$-\int_{\mathbb{R}^{n} \times \Omega} \phi d\mu  - g(\phi)= -\int_{\mathbb{R}^{n} \times \Omega} \phi d\mu - (- \infty)=+ \infty$$
thus $g^*(\mu)=0$.

Otherwise, if $ \mu \not\in \Mm_{2}$ there exists $\phi(v, \omega)=v \cdot D_{x}\varphi(0, \omega) \in  \mathcal{C}$ such that $\int_{\mathbb{R}^{n} \times \Omega} \phi d\mu >0$. Define $\phi_{\lambda}=\lambda v \cdot D_{x} \varphi(0, \omega) \in  \mathcal{C}$ then
$$g^*(\mu) \leq   \left (- \int_{\mathbb{R}^{n} \times \Omega} \phi_{\lambda} d\mu  - g(\phi_{\lambda}) \right )= - \lambda \int_{\mathbb{R}^{n} \times \Omega} \phi d\mu  \to  - \infty$$
when $\lambda \to  + \infty$.
\end{proof}

\begin{remark}\label{matherequivalent}
Observe that
$$\displaystyle \min_{\Mm_{1} \cap \Mm_{2}}  \int_{\mathbb{R}^{n} \times \Omega} L(0, v, \omega) d\mu (v, \omega) =
\min_{\Mm } ( h^{*}(\mu) - g^{*}(\mu)).$$
In fact, 
$$
h^{*}(\mu) - g^{*}(\mu)=
\begin{cases}
\int L(0, v, \omega) d\mu(v, \omega)  - 0 & \text{if} \quad \mu \in \Mm_{1} \cap \Mm_{2} \\
\int L(0, v, \omega) d\mu(v, \omega)  - (-\infty) & \text{if} \quad \mu \in \Mm_{1}\,  \text{and}\, \mu \not\in \Mm_{2} \\ 
+\infty  - (0) & \text{if} \quad \mu \not\in \Mm_{1}\,  \text{and}\, \mu \in \Mm_{2} \\ 
+\infty  - (-\infty) & \text{if} \quad \mu \not\in \Mm_{1}\,  \text{and}\, \mu  \not\in \Mm_{2}.\\ 
\end{cases}
$$
\end{remark}

\begin{lemma}\label{continuity} 
The function $$h(\phi)=\sup_{ \mathbb{R}^{n} \times \Omega} (  \phi(v, \omega) - L(0, v, \omega))$$  is continuous.
\end {lemma}
\begin{proof}
Let $\phi_{0}$ be an arbitrary, but fixed, function in $C_\gamma^0(\mathbb{R}^{n} \times \Omega)$. Suppose $\phi_{n} \to \phi_{0}$, that is $\|\phi_n-\phi_0 \|_\gamma\to 0$. 
Let $B_{\varepsilon}(\phi_{0})=\left\lbrace \phi \in C_\gamma^0(\mathbb{R}^{n} \times \Omega) \; \vert \; \|\phi_{n} -\phi \|_\gamma < \varepsilon \right\rbrace$  be the ball of radius $\varepsilon$ centered in $\phi_{0}$. Take $\phi \in B_{\varepsilon}(\phi_{0})$. Since,
$\lim_{\vert v \vert \to \infty} \frac{\vert v \vert}{\gamma(v)} =0$, $\lim_{\vert v \vert \to \infty} \frac{L(0, v, \omega)}{\gamma(v)} =+\infty $ and  $\lim_{|v|\rightarrow \infty}\frac{ \vert \phi_{0}(v, \omega)  \vert}{\gamma(v)}\rightarrow 0$ uniformly on $\omega \in \Omega$, given $\delta, M >0$, there exists $R >0$ such that 
$$
\begin{cases}
 \frac{\gamma(v)}{\vert v \vert} >\frac{1}{\delta}  \quad & \text{if} \quad \vert v \vert \quad>R \\
\vert \frac{\phi_{0}(v, \omega)}{\gamma(v) }  \vert < \delta  \quad & \text{if} \quad \vert v \vert \quad>R \\
\frac{L(0, v, \omega)}{ \gamma(v)} > M  \quad & \text{if} \quad \vert v \vert \quad>R. \\
\end{cases}
$$
Then, for $|v|>R$, 
$$ -\phi(v, \omega) - L(0, v, \omega)=  \left\lbrace  \frac{-\phi(v, \omega) + \phi_{0}(v, \omega)}{\gamma(v)}  + \frac{- \phi_{0}(v, \omega) }{\gamma(v)}  - \frac{L(0, v, \omega)}{\gamma(v)}\right\rbrace \gamma(v)< $$ 
$$< \frac{\gamma(v)}{\vert v \vert} (\varepsilon+\delta - M) \vert v \vert   \to  - \infty$$
when $\vert v \vert \to +\infty$.

As $$\|\phi_{n} -\phi_{0} \|_\gamma=\sup_{\mathbb{R}^{n} \times \Omega } \frac{|(\phi_{n} -\phi_{0})(v, \omega)|}{\gamma(v)} \to 0$$ we have that, for $n$ big
 enough, we can choose $R$ in such way that
$$h(\phi)=\sup_{ \{ \vert v \vert \leq R  \} \times \Omega} (  -\phi(v, \omega) - L(0, v, \omega)),$$
and 
$$h(\phi_{n})=\sup_{ \{ \vert v \vert \leq R  \} \times \Omega} ( - \phi_{n}(v, \omega) - L(0, v, \omega)).$$

Since the convergence $-\phi_{n} -L(0, v, \omega) \to -\phi_{0}  -L(0, v, \omega)$ is uniform on the compact $\{ \vert v \vert \leq R  \} \times \Omega$, we have 
$$\lim_{n \to \infty} h(\phi_{n})=\lim_{n \to \infty} \sup_{ \{ \vert v \vert \leq R  \} \times \Omega} (  -\phi_{n}(v, \omega) - L(0, v, \omega))=$$
$$=\sup_{ \{ \vert v \vert \leq R  \} \times \Omega}\lim_{n \to \infty}   (- \phi_{n}(v, \omega) - L(0, v, \omega))=h(\phi_{0}).$$
Thus the lemma is proved.
\end{proof}

The last ingredient of the duality is the Legendre-Fenchel-Rockafellar Theorem, see for instance \cite{Vil}.

\begin{theorem}\label{LegFenchRock} (Legendre-Fenchel-Rockafellar) 
Let $E$ be a locally convex Hausdorff topological vectorial space over $\mathbb{R}$ with dual $E^{*} $. Suppose  that $h:E \to (-\infty, +\infty ]$ is convex and lower semicontinuous and $g: E^{*}  \to [-\infty, +\infty ) $ is concave and upper semicontinuous. Then
$$ \min_{E^{*} } ( h^{*}  - g^{*} )  = \sup_{E } (g -h),$$
provided that $h$ or $g$ is continuous at some point where both functions are finite. It is part of the theorem that the left hand side is a minimum.
\end {theorem}

\begin{lemma} \label{UnifContOfg-f}
Define the functional, $S(\phi)= g(\phi) - h(\phi) $. Then $S$ is uniformly continuous in  the interior of  $\mathcal{C}$.
\end{lemma}
\begin{proof}
In fact, given $\varepsilon>0$, if  $\|\phi_{1} -\phi_{2}  \|_\gamma < \varepsilon$, that is,
$ -\varepsilon \gamma(v) < \phi_{1}(v,w) -\phi_{2}(v,w) < \varepsilon \gamma(v) $, for all $(v,w)$, then
$$ | S(\phi_{1}) -S(\phi_{2}) |< \left\lbrace \inf \gamma(v) \right\rbrace \varepsilon.$$
In particular $$\sup_{\phi \in \mathcal{C} } g(\phi) - h(\phi) = \sup_{
\begin{matrix}
^{\phi= v \cdot D_{x}\varphi(0, \omega)} \\
^{\varphi \in C^{1}_{s} } \\
\end{matrix}
} g(\phi) - h(\phi).$$
\end{proof}

\section{Proof of Proposition \ref{solviscalpha} }
\label{Appendix3}

%\textbf{Proof of Proposition \ref{solviscalpha}: }
\begin{proof}[Proof of Proposition \ref{solviscalpha}]
We must to prove that the function $u_{\alpha}$ is stationary. 
Since $L \geq 0$,  $u_{\alpha}$ is well defined as an infimum. On the other hand the stationarity
is an easy consequence of the correspondence between the set of all globally Lipschitz trajectories with initial condition $x(0)=x$ and the set of all globally Lipschitz trajectories with initial condition $y(0)=0$, given by, $\lbrace x(t) \rbrace \to \lbrace y(t)=x(t)-x \rbrace$.
In fact,
\begin{align*}
u_{\alpha}(0,\tau_{x}\omega)&=\inf_{y(0)=0} \int_{0}^{+\infty} e^{- \alpha t} L(y(t), \dot{y}(t), \tau_{x}\omega) dt 
%\\&=
%\inf_{x(0)=x} \int_{0}^{+\infty} e^{- \alpha t} L(x(t)-x, \dot{x}(t), \tau_{x}\omega) dt 
\\&=
\inf_{x(0)=x} \int_{0}^{+\infty} e^{- \alpha t} L((x(t)-x)+x, \dot{x}(t), \omega) dt =u_{\alpha}(x, \omega).
\end{align*}

In order to prove that $u_{\alpha}$ is a viscosity solution in $\omega$,
let  $\varphi : \mathbb{R}^{n} \times \Omega \to \mathbb{R}$  be a stationary function such that $u_{\alpha}(0, \omega)-\varphi(0, \omega)$ has a local minimum (resp. maximum) in $\omega_{\varphi} \in \Omega$ and $u_{\alpha}(0, \omega_{\varphi})-\varphi(0, \omega_{\varphi})=0$. 

Consider a trajectory satisfying $x(0)=0$ such that $x(t)$  is a finite time minimizing, globally Lipschitz trajectory,  for the dynamic programing principle 
\eqref{DynProgPrinciple}, that is, 
\begin{equation}
u_{\alpha}(0, \omega_{\varphi})=   \int_{0}^{T} e^{- \alpha t} L(x(t), \dot{x}(t), \omega_{\varphi}) dt + 
e^{- \alpha T} u_{\alpha}(x(T), \omega_{\varphi}), \label{(1b)}
\end{equation}
for $T$ small enough.

Suppose that $\mathcal{H}(\varphi , \omega_{\varphi}) < 0$, by continuity there is a neighborhood $B$ of $\omega_{\varphi}$ in $\Omega$ and $\delta >0$ such that
$\mathcal{H}(\varphi , \omega ) < - \delta $
for all $\omega \in B$. Since $\mathcal{H}(\varphi , \omega )= H(0, D_{x}\varphi(0, \omega)  , \omega) + \alpha \varphi (0, \omega)$ we have
$-v D_{x}\varphi(0, \omega) - L(0, v, \omega) + \alpha \varphi (0, \omega) < - \delta,$
for all $\omega \in B$ and $v \in \mathbb{R}^{n}$. 
If we choose $v=\dot{x}(t)$ and $\omega= \tau_{x(t)} \omega_{\varphi}$ then
$$\dot{x}(t) D_{x}\varphi(0 , \tau_{x(t)} \omega_{\varphi}) +L(x(t), \dot{x}(t), \omega_{\varphi}) -  \alpha \varphi (x(t), \omega_{\varphi}) >  \delta,$$
for $0<t<T$.

Integrating this expression and using $\frac{d \,}{dt}\varphi (x(t), \omega)= \dot{x}(t) D_{x} \varphi (0,\tau_{x(t)} \omega)$ we get,
$$ \varphi(0, \tau_{x(T)}\omega_{\varphi}) - \varphi(0, \omega_{\varphi}) + \int_{0}^{T} L(x(t), \dot{x}(t), \omega_{\varphi}) dt -  \alpha \int_{0}^{T}  \varphi (x(t), \omega_{\varphi}) dt >  \delta T.$$
Since $u_{\alpha}(0, \omega) \geq \varphi(0, \omega)$ in $B$ and $u_{\alpha}(0, \omega_{\varphi})=\varphi(0, \omega_{\varphi})$, we have
$$ u_{\alpha}(0, \tau_{x(T)}\omega_{\varphi}) - u_{\alpha}(0, \omega_{\varphi}) + \int_{0}^{T} L(x(t), \dot{x}(t), \omega_{\varphi}) dt -  \alpha \int_{0}^{T}  \varphi (x(t), \omega_{\varphi}) dt >  \delta T.$$
Using \eqref{(1b)}  in the last inequality we get,
$$(1-e^{- \alpha T}) u_{\alpha}(0, \tau_{x(T)}\omega_{\varphi}) +  \int_{0}^{T} (1-e^{- \alpha t}) L(x(t), \dot{x}(t), \omega_{\varphi}) dt -  \alpha \int_{0}^{T}  \varphi (x(t), \omega_{\varphi}) dt >  \delta T.$$
Writing 
$$ u_{\alpha}(0, \tau_{x(T)}\omega_{\varphi}) +\frac{T}{(1-e^{- \alpha T})} \frac{1}{T} \int_{0}^{T} (1-e^{- \alpha t}) L(x(t), \dot{x}(t), \omega_{\varphi}) dt - $$ $$ \alpha \frac{T}{(1-e^{- \alpha T})} \frac{1}{T}  \int_{0}^{T}  \varphi (x(t), \omega_{\varphi}) dt >  \delta \frac{T}{1-e^{- \alpha T}}$$
and using $\displaystyle \lim_{T \to 0} \frac{T}{1-e^{- \alpha T}} =\frac{1}{\alpha}$, we get
$$ u_{\alpha}(0,  \omega_{\varphi}) - \varphi (0, \omega_{\varphi})  > \frac{ \delta}{\alpha} $$
contradicting $u_{\alpha}(0, \omega_{\varphi})=\varphi(0, \omega_{\varphi})$.

The proof for the maximum case is analogous and so the theorem is proved.
\end{proof}

\section{Proof of Theorem \ref{LipschitzOfViscSol}}
\label{apC}

In this last appendix we give a proof of Theorem \ref{LipschitzOfViscSol}. Before that we need to establish some additional results. We 
note here that we will be using the techniques in \cite{EGom1} (see also \cite{BG}) adapted to the stationary setting.

\begin{remark}\label{SolViscLipschitz}  
Let $u_{\alpha}$ be a viscosity solution in $\omega$ of \eqref{DiscHJE}, then, because it is also a viscosity solution in $x$ (Proposition \ref{ViscSolEquiv}) and it is Lipschitz,
$D_{x}u_{\alpha}(0, \tau_{y}\omega)$ is defined Lebesgue almost everywhere and
$$ H(y, D_{x}u_{\alpha}(y, \omega), \omega) + \alpha u_{\alpha}(y, \omega)=0,$$
for Lebesgue almost everywhere $y \in \mathbb{R}^{n}$. 
\end {remark}
%\begin{proof}
%The first part is a consequence of  that claims that  $u_{\alpha}$ is a viscosity solution in $x$. By Theorem \ref{EstabilityViscositySol},  $u_{\alpha}(\cdot, \omega)$ is Lipschitz. Thus $$ H(y, D_{x}u_{\alpha}(y, \omega), \omega) + \alpha u_{\alpha}(y, \omega)=0,$$
%Lebesgue almost everywhere $y \in \mathbb{R}^{n}$, see for instance \cite{Bardi}.% or \cite{BG}, Proposition 10.
%Using Remark \ref{viscsolsmooth2} we conclude that $D_{x}u^{\varepsilon}_{\alpha}(0, \omega) =\int_{\mathbb{R}^{n}} D_{x}u_{\alpha}(0, \tau_{y}\omega) \eta^{\varepsilon}(y)  dy.$
%\end{proof}

For any probability  measure  $\mu$, we can define a new measure of probability $\tilde{\mu}$ in  $\mathbb{R}^{n} \times \Omega$ given by, 
$$\int_{\mathbb{R}^{n} \times \Omega}  \psi(p, \omega)  d\tilde{\mu}(p, \omega) = \int_{\mathbb{R}^{n} \times \Omega}  \psi(- D_vL (0, v, \omega) , \omega)  d\mu(v, \omega) .$$
In this case, the integral holonomy constraint can be rewritten as
$$\int_{\mathbb{R}^{n} \times \Omega}  D_pH(0, p, \omega) \cdot D_{x}\varphi(0, \omega) d\tilde{\mu}(p, \omega)=0,$$ $\forall  \varphi \in C^{1}_{s}(\mathbb{R}^{n}\times\Omega)$.
%Additionally, we define the projection $\theta$ in $\Omega$ of $\mu$,  by
%$$ \int_{\Omega}  f(\omega)d\theta(\omega) = \int_{\mathbb{R}^{n} \times \Omega}  f(\omega)  d\mu(v, \omega).$$ Obviously, the projections of $\mu$ and %$\tilde{\mu}$ are the same, so we denote $\theta$ for both probabilities.

\begin{theorem}\label{DifViscSolut}  
Let $\mu_{\alpha}$ be a holonomic discounted stationary Mather measure. Denote the projection in the coordinate $\omega$ of $\mu_\alpha$ by
 $\theta_\alpha$, that is 
 \[
 \int_{\Omega} \varphi(\omega) d\theta_\alpha=\int_{\Rr^n\times \Omega} \varphi(\omega) d\mu_\alpha.
 \] 
 If $u_{\alpha}$ is a viscosity solution of \eqref{DiscHJE},  then $D_{x}u_{\alpha}(0, \omega)$ exists $\theta_\alpha$-a.e,
 and $\tilde{\mu}_{\alpha}$-a.e,  $p=-D_{x}u_{\alpha}(0, \omega)$.
\end {theorem}
\begin{proof}
By the strict uniform continuity of $H$ there exists $\gamma >0$  such that for any $p, q, y  \in \mathbb{R}^{n}$ and $\omega \in \Omega$ we have 
$$H(0, p, \tau_{y}\omega ) \geq H(0, q, \tau_{y}\omega )  + D_pH(0, q, \tau_{y} \omega) (p-q) + \frac{\gamma}{2} |p-q|^{2}.$$
Let $u^{\varepsilon}=u* \eta$,  by Remark \ref{SolViscLipschitz}, for almost every $\omega$ and $y$, let $p=D_{x}u_{\alpha}(0, \tau_{y}\omega) $ and $q=D_{x}u_{\alpha}^{\varepsilon}(0, \omega)$. Then 
\begin{align*}
&H(0, D_{x}u_{\alpha}(0, \tau_{y}\omega) , \tau_{y}\omega ) \geq H(0, D_{x}u_{\alpha}^{\varepsilon}(0, \omega), \tau_{y}\omega )  \\&\quad +
 D_pH(0, D_{x}u_{\alpha}^{\varepsilon}(0, \omega), \tau_{y} \omega) (D_{x}u_{\alpha}(0, \tau_{y}\omega)-D_{x}u_{\alpha}^{\varepsilon}(0, \omega))+ \frac{\gamma}{2} |D_{x}u_{\alpha}(0, \tau_{y}\omega)-D_{x}u_{\alpha}^{\varepsilon}(0, \omega)|^{2}.\end{align*}
Multiplying by $\eta^{\varepsilon}(y)$ and integrating we get
\begin{align*}
&\int_{\mathbb{R}^{n}} H(0, D_{x}u_{\alpha}^{\varepsilon}(0, \omega), \tau_{y}\omega ) \eta^{\varepsilon}(y) dy  + \int_{\mathbb{R}^{n}} \frac{\gamma}{2} |D_{x}u_{\alpha}(0, \tau_{y}\omega)-D_{x}u_{\alpha}^{\varepsilon}(0, \omega)|^{2} \eta^{\varepsilon}(y) dy \\ &\leq 
\int_{\mathbb{R}^{n}} H(0, D_{x}u_{\alpha}(0, \tau_{y}\omega), \tau_{y}\omega ) \eta^{\varepsilon}(y) dy  \\ &+  \int_{\mathbb{R}^{n}} D_p H(0, D_{x}u_{\alpha}^{\varepsilon}(0, \omega), \tau_{y} \omega)  \left[ D_{x}u_{\alpha}^{\varepsilon}(0, \omega)-D_{x}u_{\alpha}(0, \tau_{y}\omega)\right]  \eta^{\varepsilon}(y) dy.
% \\ &= - \alpha u_{\alpha}^{\varepsilon}(0, \omega) +  o_{\omega}(\varepsilon),
\end{align*}
Remark \ref{SolViscLipschitz} implies that, $ H(y, D_{x}u_{\alpha}(y, \omega), \omega) = - \alpha u_{\alpha}(y, \omega)$  almost everywhere $y$. Thus 
\begin{equation}
\label{um}
\int_{\mathbb{R}^{n}} H(0, D_{x}u_{\alpha}^{\varepsilon}(0, \tau_{y}\omega), \tau_{y}\omega ) \eta^{\varepsilon}(y) dy  + \beta^{\varepsilon}(\omega)  \leq   - \alpha u_{\alpha}^{\varepsilon}(0, \omega) +  o_{\omega}(\varepsilon)
\end{equation}
where
$$\beta^{\varepsilon}(\omega)=\int_{\mathbb{R}^{n}} \frac{\gamma}{4} |D_{x}u_{\alpha}(0, \tau_{y}\omega)-D_{x}u_{\alpha}^{\varepsilon}(0, \tau_{y}\omega)|^{2} \eta^{\varepsilon}(y) dy.$$

On the other hand, the convexity of $H$, implies that,
\begin{align}
\notag
&\int_{\mathbb{R}^{n} \times \Omega} \frac{\gamma}{2} |D_{x}u_{\alpha}^{\varepsilon}(0, \omega)- p|^{2} d\tilde{\mu}_{\alpha}(p, \omega)\\ \notag
&\leq  \int_{\mathbb{R}^{n} \times \Omega} \left[H(0, D_{x}u_{\alpha}^{\varepsilon}(0, \omega) , \omega ) -  H(0, p, \omega )  -  D_pH(0, p,  \omega) (D_{x}u_{\alpha}^{\varepsilon}(0, \omega)- p) \right]d\tilde{\mu}_{\alpha}(p, \omega)\\\notag &=  \int_{\mathbb{R}^{n} \times \Omega} H(0, D_{x}u_{\alpha}^{\varepsilon}(0, \omega) , \omega ) d\tilde{\mu}_{\alpha}(p, \omega) \\\notag&-\int_{\mathbb{R}^{n} \times \Omega} \left[ H(0, p, \omega ) +  D_pH(0, p,  \omega) D_{x}u_{\alpha}^{\varepsilon}(0, \omega)- D_pH(0, p,  \omega) p \right]  d\tilde{\mu}_{\alpha}(p, \omega) \\\notag &=  \int_{\mathbb{R}^{n} \times \Omega} H(0, D_{x}u_{\alpha}^{\varepsilon}(0, \omega) , \omega ) + L(0,- D_pH(0, p, \omega), \omega) d\tilde{\mu}_{\alpha}(p, \omega) \\\label{dois}&  =  \int_{\mathbb{R}^{n} \times \Omega} H(0, D_{x}u_{\alpha}^{\varepsilon}(0, \omega) , \omega )  d\tilde{\mu}_{\alpha}(p, \omega) +  \bar{H}_{\alpha}. 
\end{align}

Integrating \eqref{um} with respect to $\tilde{\mu}$ and adding \eqref{dois}, we get
$$ \int_{\mathbb{R}^{n} \times \Omega}  \frac{\gamma}{2} |D_{x}u_{\alpha}^{\varepsilon}(0, \omega)- p|^{2} d\tilde{\mu}_{\alpha}(p, \omega) + \int_{\Omega} \beta^{\varepsilon}(\omega) d\theta(\omega) < o(\varepsilon).$$
So, $\theta_\alpha$ almost everywhere we have $\displaystyle D_{x}u_{\alpha}(0, \omega) = \lim_{\varepsilon \to 0} D_{x}u_{\alpha}^{\varepsilon}(0,  \omega)$, in particular $p=D_{x}u_{\alpha}(0, \omega)$ in the support of $\tilde{\mu}$.
\end{proof}

%The next corollary is a immediate consequence of the last part of the proof of the previous Theorem.
%% \ref{DifViscSolut***}.
%\begin{corollary}\label{DifViscSolSmooth}  
%Let  $u_{\alpha}^{\varepsilon}$ be a smoothing of $u_{\alpha}$ (Remark \ref{viscsolsmooth}). Then, 
%$$\int_{ \Omega} | D_{x}u_{\alpha}^{\varepsilon}(0, \omega) - D_{x}u_{\alpha}(0, \omega) |^{2}   d\theta_\alpha(\omega) < C \cdot \varepsilon.$$ 
%\end {corollary}

%{\bf Check next lemma, maybe we don't need it or at least set it as a remark}
%
%
%\begin{corollary}\label{DifViscSolutREFORM}  
%Let $\mu_{\alpha}$ be an invariant  Mather measure  for the discounted stationary Mather problem with trace $\nu_{\alpha}$.  If $u_{\alpha}$ is a viscosity solution of %\eqref{DiscHJE},  then $D_{x}u_{\alpha}(0, \omega)$ exists $\theta$-a.e, and 
%$$ D_vL (0, v, \omega) = D_{x}u_{\alpha}(0, \omega)  \text{ and } D_xL (0, v, \omega) =  D_x H(0, D_{x}u_{\alpha}(0, \omega), \omega),$$
%$\mu_{\alpha}$-a.e.
%\end {corollary}
%\begin{proof} Applying Theorem \ref{DifViscSolut} and $ -p =  D_vL (0, - H_{p}(0, p, \omega), \omega)$  we get,
%\begin{align*}
%D_vL (0, v, \omega) &=D_vL (0, - D_pH(0, p, \omega), \omega) \\
%&=D_vL (0, - D_pH(0, -D_{x}u_{\alpha}(0, \omega), \omega), \omega) =D_{x}u_{\alpha}(0, \omega).
%\end{align*}
%
%On the other hand
%\begin{align*}
% D_xL (0, v, \omega) &= - D_xH(0, p, \omega)\\
% &=- D_xH(0, -D_{x}u_{\alpha}(0, \omega), \omega)= D_xH(0, D_{x}u_{\alpha}(0, \omega), \omega).
%\end{align*}
%\end{proof}

\begin{theorem}\label{SemicOfViscSol}
Let $\mu_{\alpha}$ be a holonomic Mather measure for the discounted stationary Mather problem.
If $u_{\alpha}$ is a viscosity solution of \eqref{DiscHJE}, then for each $h \in \mathbb{R}$, 
$$| u_{\alpha}(h, \omega)   - 2  u_{\alpha}(0, \omega) + u_{\alpha}(-h, \omega) | \leq  C  | h |^{2},$$
 $\theta$ almost everywhere.
\end {theorem}
\begin{proof}
If $h \neq 0$ then we define,
$$\tilde{u}_{\alpha}(x, \omega)=u_{\alpha}(x+h, \omega)  \text{ and }  \hat{u}_{\alpha}(x, \omega)=u_{\alpha}(x-h, \omega),$$
and
$ \tilde{u}_{\alpha}^{\varepsilon}(x, \omega)  \text{ and }  \hat{u}_{\alpha}^{\varepsilon}(x, \omega),$
the corresponding smoothings (see Remark \ref{viscsolsmooth}).

Remember that 
$$H(h, D_{x}\tilde{u}_{\alpha}^{\varepsilon}(0, \omega),  \omega ) + \alpha \tilde{u}_{\alpha}^{\varepsilon}(0, \omega) \leq c \varepsilon,$$ 
and
$$H(-h, D_{x}\hat{u}_{\alpha}^{\varepsilon}(0, \omega),  \omega ) + \alpha \hat{u}_{\alpha}^{\varepsilon}(0, \omega) \leq c \varepsilon.$$ 
Thus,
\begin{align}\notag
H(0, & D_{x}\tilde{u}_{\alpha}^{\varepsilon}(0, \omega) , \omega ) - 2  H(0, D_{x}u_{\alpha}(0, \omega), \omega ) + H(0, D_{x}\hat{u}_{\alpha}^{\varepsilon}(0, \omega) , \omega )
\\\notag
=&H(0, D_{x}\tilde{u}_{\alpha}^{\varepsilon}(0, \omega) , \omega ) - H(h, D_{x}\tilde{u}_{\alpha}^{\varepsilon}(0, \omega) , \omega )+ H(h, D_{x}\tilde{u}_{\alpha}^{\varepsilon}(0, \omega) , \omega ) + \alpha \tilde{u}_{\alpha}^{\varepsilon}(0, \omega)
\\\notag& -  \alpha \tilde{u}_{\alpha}^{\varepsilon}(0, \omega)
+ 2   \alpha u_{\alpha}(0, \omega)  -  \alpha \hat{u}_{\alpha}^{\varepsilon}(0, \omega)
\\\notag& +  \alpha \hat{u}_{\alpha}^{\varepsilon}(0, \omega) + H(-h, D_{x}\hat{u}_{\alpha}^{\varepsilon}(0, \omega) , \omega )
- H(-h, D_{x}\hat{u}_{\alpha}^{\varepsilon}(0, \omega) , \omega )+
H(0, D_{x}\hat{u}_{\alpha}^{\varepsilon}(0, \omega) , \omega )
\\ \notag 
\leq &2 c  \varepsilon  - \alpha \left(  \tilde{u}_{\alpha}^{\varepsilon}(0, \omega)
- 2     u_{\alpha}(0, \omega)  +    \hat{u}_{\alpha}^{\varepsilon}(0, \omega) \right)
\\\label{uno}& - 
\left( D_xH(0, D_{x}\tilde{u}_{\alpha}^{\varepsilon}(0, \omega) , \omega ) - D_xH(0, D_{x}\hat{u}_{\alpha}^{\varepsilon}(0, \omega) , \omega )\right) h + O(|h|^2).
\end{align}  
On the other hand the convexity of $H$ implies that
\begin{align*}
H(0, D_{x}\tilde{u}_{\alpha}^{\varepsilon}(0, \omega) , \omega ) \geq &H(0, D_{x}u_{\alpha}(0, \omega), \omega ) \\ & +  D_pH(0, D_{x}u_{\alpha}(0, \omega),   \omega) (D_{x}\tilde{u}_{\alpha}^{\varepsilon}(0,  \omega) - D_{x}u_{\alpha}(0,  \omega))\\& +  \frac{\gamma}{2} |D_{x}\tilde{u}_{\alpha}^{\varepsilon}(0,  \omega) - D_{x}u_{\alpha}(0,  \omega)|^{2},
\end{align*}
and
\begin{align*}
H(0, D_{x}\hat{u}_{\alpha}^{\varepsilon}(0, \omega) , \omega ) \geq &H(0, D_{x}u_{\alpha}(0, \omega), \omega ) \\& + D_pH(0, D_{x}u_{\alpha}(0, \omega),   \omega) (D_{x}\hat{u}_{\alpha}^{\varepsilon}(0,  \omega) - D_{x}u_{\alpha}(0,  \omega))\\& +  \frac{\gamma}{2} |D_{x}\hat{u}_{\alpha}^{\varepsilon}(0,  \omega) - D_{x}u_{\alpha}(0,  \omega)|^{2}.
\end{align*}

Adding these two formulas we obtain the following inequality:
\begin{align*}
&\frac{\gamma}{2} \left(  |D_{x}\tilde{u}_{\alpha}^{\varepsilon}(0,  \omega) - D_{x}u_{\alpha}(0,  \omega)|^{2} +  |D_{x}\hat{u}_{\alpha}^{\varepsilon}(0,  \omega) - D_{x}u_{\alpha}(0,  \omega)|^{2} \right) \\  &+ D_pH(0, D_{x}u_{\alpha}(0, \omega),   \omega) (D_{x}\tilde{u}_{\alpha}^{\varepsilon}(0,  \omega) - 2 D_{x}u_{\alpha}(0,  \omega) + D_{x}\hat{u}_{\alpha}^{\varepsilon}(0,  \omega)) \\&\leq 
\left( H(0, D_{x}\tilde{u}_{\alpha}^{\varepsilon}(0, \omega) , \omega ) - 2  H(0, D_{x}u_{\alpha}(0, \omega), \omega ) + H(0, D_{x}\hat{u}_{\alpha}^{\varepsilon}(0, \omega) , \omega ) \right) .
\end{align*}
By \eqref{uno} we have,
\begin{align*}
& \frac{\gamma}{2} \left(  |D_{x}\tilde{u}_{\alpha}^{\varepsilon}(0,  \omega) - D_{x}u_{\alpha}(0,  \omega)|^{2} +  |D_{x}\hat{u}_{\alpha}^{\varepsilon}(0,  \omega) - D_{x}u_{\alpha}(0,  \omega)|^{2} \right) \\ & +  D_pH(0, D_{x}u_{\alpha}(0, \omega),   \omega) (D_{x}\tilde{u}_{\alpha}^{\varepsilon}(0,  \omega) - 2 D_{x}u_{\alpha}(0,  \omega) + D_{x}\hat{u}_{\alpha}^{\varepsilon}(0,  \omega)) \\ &\leq 2 c  \varepsilon  - \alpha \left(  \tilde{u}_{\alpha}^{\varepsilon}(0, \omega)
- 2     u_{\alpha}(0, \omega)  +    \hat{u}_{\alpha}^{\varepsilon}(0, \omega) \right)\\ & -
\left( D_xH(0, D_{x}\tilde{u}_{\alpha}^{\varepsilon}(0, \omega) , \omega ) - D_xH(0, D_{x}\hat{u}_{\alpha}^{\varepsilon}(0, \omega) , \omega )\right) h+O(|h|^2).
\end{align*}  
Or equivalently,
\begin{align}
\notag
&\notag\frac{\gamma}{2} \left(  |D_{x}\tilde{u}_{\alpha}^{\varepsilon}(0,  \omega) - D_{x}u_{\alpha}(0,  \omega)|^{2} +  |D_{x}\hat{u}_{\alpha}^{\varepsilon}(0,  \omega) - D_{x}u_{\alpha}(0,  \omega)|^{2} \right)  \\ &\notag + D_pH(0, D_{x}u_{\alpha}(0, \omega),   \omega) (D_{x}\tilde{u}_{\alpha}^{\varepsilon}(0,  \omega) - 2 D_{x}u_{\alpha}(0,  \omega) + D_{x}\hat{u}_{\alpha}^{\varepsilon}(0,  \omega))\\&\notag +\alpha \left(  \tilde{u}_{\alpha}^{\varepsilon}(0, \omega) - 2     u_{\alpha}(0, \omega)  +    \hat{u}_{\alpha}^{\varepsilon}(0, \omega) \right)\\&\label{two2} \leq 2 c  \varepsilon  -  \left( D_xH(0, D_{x}\tilde{u}_{\alpha}^{\varepsilon}(0, \omega) , \omega ) - D_xH(0, D_{x}\hat{u}_{\alpha}^{\varepsilon}(0, \omega) , \omega )\right) (h)+O(|h|^2).
\end{align}  
Define, $\beta^{\varepsilon}(x, \omega)=  \tilde{u}_{\alpha}^{\varepsilon}(x, \omega) - 2     u_{\alpha}(x, \omega)  + \hat{u}_{\alpha}^{\varepsilon}(x, \omega)$, so \eqref{two2} can be rewritten as
\begin{align}
\notag
& \frac{\gamma}{2} \left(  |D_{x}\tilde{u}_{\alpha}^{\varepsilon}(0,  \omega) - D_{x}u_{\alpha}(0,  \omega)|^{2} +  |D_{x}\hat{u}_{\alpha}^{\varepsilon}(0,  \omega) - D_{x}u_{\alpha}(0,  \omega)|^{2} \right) \\&\notag  +  D_pH(0, D_{x}u_{\alpha}(0, \omega),   \omega)  D_{x} \beta^{\varepsilon}(0, \omega)   + \alpha \beta^{\varepsilon}(0, \omega) \\&\label{three3}  \leq  2 c  \varepsilon  -  \left( D_xH(0, D_{x}\tilde{u}_{\alpha}^{\varepsilon}(0, \omega) , \omega ) - D_xH(0, D_{x}\hat{u}_{\alpha}^{\varepsilon}(0, \omega) , \omega )\right) ( h)+O(|h|^2)
\end{align}
Applying the inequality
\begin{align*}
\|
&\left(
D_x H(0, D_{x}\tilde{u}_{\alpha}^{\varepsilon}(0, \omega) , \omega ) - D_xH(0, D_{x}\hat{u}_{\alpha}^{\varepsilon}(0, \omega) , \omega ) 
\right) h\|
 \\\leq  &\| D^2{px}H(0, D_{x}\tilde{u}_{\alpha}^{\varepsilon}(0, \omega) , \omega ) \| \cdot | D_{x}\tilde{u}_{\alpha}^{\varepsilon}(0, \omega) - D_{x}\hat{u}_{\alpha}^{\varepsilon}(0, \omega) | \cdot  |h| \\\leq & \frac{\gamma}{4} | D_{x}\tilde{u}_{\alpha}^{\varepsilon}(0, \omega) - D_{x}\hat{u}_{\alpha}^{\varepsilon}(0, \omega) |^{2} + \frac{1}{\gamma} | D_{x}\tilde{u}_{\alpha}^{\varepsilon}(0, \omega) - D_{x}\hat{u}_{\alpha}^{\varepsilon}(0, \omega) |^{2} \cdot  |h|^{2} \\\leq& \frac{\gamma}{4} \left(  |D_{x}\tilde{u}_{\alpha}^{\varepsilon}(0,  \omega) - D_{x}u_{\alpha}(0,  \omega)|^{2} +  |D_{x}\hat{u}_{\alpha}^{\varepsilon}(0,  \omega) - D_{x}u_{\alpha}(0,  \omega)|^{2} \right) \\& +  \frac{1}{\gamma} | D_{x}\tilde{u}_{\alpha}^{\varepsilon}(0, \omega) - D_{x}\hat{u}_{\alpha}^{\varepsilon}(0, \omega) |^{2} \cdot  |h|^{2},
\end{align*}
to \eqref{three3} we get, 
\begin{align}\notag & \frac{\gamma}{4} \left(  |D_{x}\tilde{u}_{\alpha}^{\varepsilon}(0,  \omega) - D_{x}u_{\alpha}(0,  \omega)|^{2} +  |D_{x}\hat{u}_{\alpha}^{\varepsilon}(0,  \omega) - D_{x}u_{\alpha}(0,  \omega)|^{2} \right)  \\&\label{quatro} +D_pH(0, D_{x}u_{\alpha}(0, \omega),   \omega)  D_{x} \beta^{\varepsilon}(0, \omega)   + \alpha \beta^{\varepsilon}(0, \omega)  \leq 2 C  ( \varepsilon   +  |h|^{2}). 
\end{align}
Consider a function $\Psi:  \mathbb{R} \to  \mathbb{R}$, such that  $\Phi (s)=\Psi'(s) \geq 0$. We can multiply \eqref{quatro}
 by $\Phi \left(\frac{\beta^{\varepsilon}(0, \omega)}{|h|^{2}}\right)$ and integrate with respect to $\theta$,
\begin{align}
&\notag\int_{ \Omega}  \frac{\gamma}{4} \left(  |D_{x}\tilde{u}_{\alpha}^{\varepsilon}(0,  \omega) - D_{x}u_{\alpha}(0,  \omega)|^{2} +  |D_{x}\hat{u}_{\alpha}^{\varepsilon}(0,  \omega) - D_{x}u_{\alpha}(0,  \omega)|^{2} \right)\Phi \left(\frac{\beta^{\varepsilon}(0, \omega)}{|h|^{2}}\right) d\theta \\\notag & +\int_{ \Omega} D_pH(0, D_{x}u_{\alpha}(0, \omega),   \omega)  D_{x} \beta^{\varepsilon}(0, \omega)\Phi \left(\frac{\beta^{\varepsilon}(0, \omega)}{|h|^{2}}\right)  d\theta \\ &\label{cinco} +\int_{ \Omega} \alpha \beta^{\varepsilon}(0, \omega) \Phi \left(\frac{\beta^{\varepsilon}(0, \omega)}{|h|^{2}}\right) d\theta \leq 2 C  ( \varepsilon   +  |h|^{2})\int_{ \Omega} \Phi \left(\frac{\beta^{\varepsilon}(0, \omega)}{|h|^{2}}\right) d\theta.
\end{align}
We have
\begin{align*}
&
 |h|^{2} \int_{ \Omega} D_pH(0, D_{x}u_{\alpha}(0, \omega),   \omega) \frac{1}{ |h|^{2}} D_{x} \beta^{\varepsilon}(0, \omega)\Phi \left(\frac{\beta^{\varepsilon}(0, \omega)}{|h|^{2}}\right)  d\theta\\ &
 =\int_{ \Omega} D_pH(0, D_{x}u_{\alpha}(0, \omega),   \omega)  D_{x} \Psi \left(\frac{\beta^{\varepsilon}}{|h|^{2}}\right)(0, \omega)   d\theta=0.
 \end{align*}
Thus, \eqref{cinco} can be restated as,
\begin{align}
\notag &\int_{ \Omega}   |D_{x}\tilde{u}_{\alpha}^{\varepsilon}(0,  \omega) - D_{x}\hat{u}_{\alpha}^{\varepsilon}(0,  \omega)|^{2} \Phi \left(\frac{\beta^{\varepsilon}(0, \omega)}{|h|^{2}}\right) d\theta \\& +   \int_{ \Omega} \alpha \beta^{\varepsilon}(0, \omega) \Phi \left(\frac{\beta^{\varepsilon}(0, \omega)}{|h|^{2}}\right) d\theta \leq 2 C  ( \varepsilon   +  |h|^{2})\int_{ \Omega} \Phi \left(\frac{\beta^{\varepsilon}(0, \omega)}{|h|^{2}}\right) d\theta. 
\label{seis}
\end{align}
Define, $A_{\lambda}=\{\omega | \frac{\beta^{\varepsilon}(0, \omega)}{|h|^{2}} \leq -\lambda \}$, and consider the  function $\Psi $ defined by
$$\Psi(s)=
\begin{cases}
s  & \text{if} \; s \leq -\lambda,  \\
1  & \text{otherwise}.
\end{cases}
$$

Fix a positive constant $\gamma$ such that the functions $\bar{u}_{\alpha} (x,  \omega)= \tilde{u}_{\alpha} (x,  \omega) - \frac{\gamma}{2} |x|^{2}$ and $\bar{u}_{\alpha}^{\varepsilon}(x,  \omega)= \tilde{u}_{\alpha}^{\varepsilon}(x,  \omega) - \frac{\gamma}{2} |x|^{2}$ are concave. Observe that a point $\omega$ is in $A_{\lambda}$ only if 
$$\bar{u}_{\alpha}^{\varepsilon}(h,  \omega) -2 \bar{u}_{\alpha} (0,  \omega) + \bar{u}_{\alpha}^{\varepsilon}(-h,  \omega) \leq -(\lambda + \gamma) |h|^{2}.$$
Define $F^{\varepsilon}(t)=\bar{u}_{\alpha}^{\varepsilon}(t \frac{h}{|h|},  \omega)$. Since $F^{\varepsilon}$ is concave and $(F^{\varepsilon})''  \leq 0$ we have 
$$\bar{u}_{\alpha}^{\varepsilon}(h,  \omega) -2 \bar{u}_{\alpha}^{\varepsilon} (0,  \omega) + \bar{u}_{\alpha}^{\varepsilon}(-h,  \omega)  \geq ( D_{x}\bar{u}_{\alpha}^{\varepsilon}(h,  \omega) - D_{x}\bar{u}_{\alpha}^{\varepsilon}(-h,  \omega) ) h.$$
Subtracting this inequalities we get,
$$(\lambda + \gamma) |h|^{2} \leq  2 | \bar{u}_{\alpha}^{\varepsilon} (0,  \omega) - \bar{u}_{\alpha}  (0,  \omega) | +  | D_{x}\bar{u}_{\alpha}^{\varepsilon}(h,  \omega) - D_{x}\bar{u}_{\alpha}^{\varepsilon}(-h,  \omega) | \, | h|.$$

Since $u_{\alpha}$ is stationary and uniformly Lipschitz continuous we have 
$| \bar{u}_{\alpha}^{\varepsilon} (0,  \omega) - \bar{u}_{\alpha}  (0,  \omega) | \leq C  \varepsilon $. thus we can choose $\varepsilon $ in such way that 
$$ | D_{x}\bar{u}_{\alpha}^{\varepsilon}(h,  \omega) - D_{x}\bar{u}_{\alpha}^{\varepsilon}(-h,  \omega) | \geq  (\frac{\lambda }{2} + \gamma) |h|$$
and
$$ | D_{x}u_{\alpha}^{\varepsilon}(h,  \omega) - D_{x}u_{\alpha}^{\varepsilon}(-h,  \omega) | \geq  (\frac{\lambda }{2} + \gamma) |h|.$$

Using this estimates in (6) we get
$$ (\frac{\lambda }{2} + \gamma)^{2} |h|^{2} \theta(A_{\lambda}) -  \alpha  \lambda |h|^{2} \theta(A_{\lambda})  \leq 2 C  ( \varepsilon   +  |h|^{2}) \theta(A_{\lambda}).$$
Observe that, if $\theta(A_{\lambda}) > 0$ then the left hand side of this inequality converges to $+ \infty$ when $\lambda \to + \infty$, so there exists a value $\lambda_{0}$ such that $\theta(A_{\lambda}) = 0$, that is, $-\lambda_{0} |h|^{2} \leq \tilde{u}_{\alpha}^{\varepsilon}(x, \omega) - 2     u_{\alpha}(x, \omega)  + \hat{u}_{\alpha}^{\varepsilon}(x, \omega)$, $\theta$ almost everywhere. The upper bound comes from the semiconcavity of $u_{\alpha}$. Thus there exists $C >0$ such that $| u_{\alpha}(h, \omega)   - 2  u_{\alpha}(0, \omega) + u_{\alpha}(-h, \omega) | \leq  C  | h |^{2} $, $\theta$ almost everywhere, which completes the proof of the theorem. 
\end {proof}

\begin{proof}[Proof of Theorem \ref{LipschitzOfViscSol}]
Let $\theta$ be the projection of $\mu_{\alpha}$. By Theorem \ref{DifViscSolut}, $D_{x}u_{\alpha}(0, \omega)$ exists $\theta$-a.e.  On the other hand, fixed $\omega \in \supp \theta$, $D_{x}u_{\alpha}(y, \omega)$ exists Lebesgue almost everywhere. 

We claim that 
$$| u_{\alpha}(y, \omega)   -  u_{\alpha}(0, \omega) + u_{\alpha}(-y, \omega)| \leq  C  | y |^{2}.$$
This claim is a consequence of Theorem \ref{SemicOfViscSol}, by choosing $h=y$ and of the semi-concavity of $u_{\alpha}$. In fact, we have
\begin{equation}
\label{ape1}
- C  | h |^{2} \leq u_{\alpha}(y, \omega)   - 2  u_{\alpha}(0, \omega) + u_{\alpha}(-y, \omega)  \leq  C  | h |^{2},
\end{equation}
\begin{equation}
\label{ape2}
 u_{\alpha}(y, \omega)   -  u_{\alpha}(0, \omega) -  D_{x}u_{\alpha}(0, \omega) y  \leq  C  | y |^{2},
\end{equation}
and
\begin{equation}
\label{ape3}
 u_{\alpha}(-y, \omega)   -  u_{\alpha}(0, \omega) + D_{x}u_{\alpha}(0, \omega) y  \leq  C  | y |^{2}. 
 \end{equation}
The claim is obtained from \eqref{ape2} and from the difference between \eqref{ape1} and \eqref{ape3}.

Let $z \in \mathbb{R}$ be a point such that $|z| \leq 2 |y|$.
The semi-concavity of $u_{\alpha}$ implies that, 
\begin{equation}
\label{ape4}
u_{\alpha}(z, \omega)   \leq  u_{\alpha}(y, \omega) + D_{x}u_{\alpha}(y, \omega) (z-y)  +  C  | z-y |^{2}.
\end{equation}
Using, $u_{\alpha}(z, \omega)  = u_{\alpha}(0, \omega) + D_{x}u_{\alpha}(0, \omega) z  +  o(| z|^{2})$ and $u_{\alpha}(y, \omega)  = u_{\alpha}(0, \omega) + D_{x}u_{\alpha}(0, \omega) y  +   o(| y|^{2})$ in (4) we get
\begin{equation}
\label{ape5}
 (D_{x}u_{\alpha}(0, \omega) -D_{x}u_{\alpha}(y, \omega))(z-y)  \leq   C  | y |^{2}.  
 \end{equation}
If we take $z=y + |y| \frac{D_{x}u_{\alpha}(0, \omega) -D_{x}u_{\alpha}(y, \omega)}{|D_{x}u_{\alpha}(0, \omega) -D_{x}u_{\alpha}(y, \omega)|}$ then we obtain $| D_{x}u_{\alpha}(y, \omega)   -  D_{x}u_{\alpha}(0, \omega) | \leq  C  | y |$.
\end{proof}

%\textbf{Proof of Proposition \ref{EstabilityViscositySol}:}

\bibliographystyle{alpha}

\bibliography{mainbibliography}

\end{document}